\documentclass[10pt,twocolumn,twoside]{IEEEtran}

\usepackage[utf8]{inputenc}
\usepackage{amsmath,amsthm,amsfonts,amssymb,bm,dsfont}
\usepackage{graphicx,cite,enumerate}
\graphicspath{{Images/}}

\usepackage{blindtext}

\newcommand{\includesvg}[2][scale=1]{\includegraphics[#1]{#2.pdf}}

\usepackage{color}

\newtheorem{theorem}{Theorem}

\newtheorem{corollary}{Corollary}
\newtheorem{proposition}{Proposition}
\theoremstyle{definition}

\newtheoremstyle{assume}
  {3pt}
  {3pt}
  {}
  {}
  {\bf}
  {}
  { }
  {\thmname{#1}.\thmnumber{#2}\thmnote{ \textnormal{(\textit{#3})}}}
\theoremstyle{assume}
\newtheorem{assumption}{A}

\DeclareMathOperator{\E}{\mathbb{E}}

\DeclareMathOperator{\trace}{Tr}

\DeclareMathOperator*{\argmin}{argmin}

\newcommand{\norm}[1]{\ensuremath{\left\| #1 \right\|}}
\newcommand{\abs}[1]{\ensuremath{{\left\vert #1 \right\vert}}}

\newcommand{\bzero}{\ensuremath{\bm{0}}}
\newcommand{\bA}{\ensuremath{\bm{A}}}

\newcommand{\bH}{\ensuremath{\bm{H}}}
\newcommand{\bI}{\ensuremath{\bm{I}}}

\newcommand{\bK}{\ensuremath{\bm{K}}}

\newcommand{\bQ}{\ensuremath{\bm{Q}}}
\newcommand{\bR}{\ensuremath{\bm{R}}}

\newcommand{\bU}{\ensuremath{\bm{U}}}

\newcommand{\bX}{\ensuremath{\bm{X}}}

\newcommand{\bp}{\ensuremath{\bm{p}}}
\newcommand{\bq}{\ensuremath{\bm{q}}}

\newcommand{\bu}{\ensuremath{\bm{u}}}

\newcommand{\bw}{\ensuremath{\bm{w}}}
\newcommand{\bx}{\ensuremath{\bm{x}}}

\newcommand{\bLambda}{\ensuremath{\bm{\Lambda}}}

\newcommand{\setN}{\ensuremath{\mathbb{N}}}

\newcommand{\setR}{\ensuremath{\mathbb{R}}}

\def\st/{\textsuperscript{st}}
\def\nd/{\textsuperscript{nd}}
\def\rd/{\textsuperscript{rd}}
\def\th/{\textsuperscript{th}}

\newcommand{\eps}{\ensuremath{\epsilon}}
\newcommand{\del}{\ensuremath{\partial}}

\newcommand{\tilw}{\ensuremath{\widetilde{\bw}}}
\newcommand{\mubar}{\ensuremath{\bar{\mu}}}

\newcommand{\fbar}{\ensuremath{\bar{f}}}
\newcommand{\fpbar}{\ensuremath{\bar{f}^{\prime}}}
\newcommand{\fppbar}{\ensuremath{\bar{f}^{\prime\prime}}}
\newcommand{\etabar}{\ensuremath{\bar{\eta}}}
\newcommand{\abar}{\ensuremath{\bar{a}}}
\newcommand{\ubar}{\ensuremath{\bm{{\hat{\bu}}}}}
\newcommand{\Kbar}{\ensuremath{\bm{{\hat{K}}}}}

\newcommand{\squishlist}{
	\begin{list}{$\bullet$}
	{ \setlength{\itemsep}{0pt}
      \setlength{\parsep}{3pt}
      \setlength{\topsep}{3pt}
      \setlength{\partopsep}{0pt}
      \setlength{\leftmargin}{1.5em}
      \setlength{\labelwidth}{1em}
      \setlength{\labelsep}{0.5em} } }

\newcommand{\squishend}{
	\end{list}  }

\title{Combination of LMS Adaptive Filters\\with Coefficients Feedback}

\author{Luiz~F.~O.~Chamon,~\IEEEmembership{Student Member,~IEEE,}
        and~Cássio~G.~Lopes,~\IEEEmembership{Senior Member,~IEEE}%
\thanks{Signal Processing Lab of the Department of Electronic Systems Engineering, University of São Paulo, Brazil. e-mail:~\mbox{\texttt{chamon@usp.br}} and \mbox{\texttt{cassio@lps.usp.br}}. Part of the results in this paper appeared in~[19].}}

\begin{document}
\maketitle
\begin{abstract}
Parallel combinations of adaptive filters have been effectively used to improve the performance of adaptive algorithms and address well-known trade-offs, such as convergence rate vs{.} steady-state error. Nevertheless, typical combinations suffer from a convergence stagnation issue due to the fact that the component filters run independently. Solutions to this issue usually involve conditional transfers of coefficients between filters, which although effective, are hard to generalize to combinations with more filters or when there is no clearly faster adaptive filter. In this work, a more natural solution is proposed by cyclically feeding back the combined coefficient vector to all component filters. Besides coping with convergence stagnation, this new topology improves tracking and supervisor stability, and bridges an important conceptual gap between combinations of adaptive filters and variable step size schemes. We analyze the steady-state, tracking, and transient performance of this topology for LMS component filters and supervisors with generic activation functions. Numerical examples are used to illustrate how coefficients feedback can improve the performance of parallel combinations at a small computational overhead.
\end{abstract}

\ifCLASSOPTIONpeerreview
	\begin{center}
		\bfseries EDICS Category: X-XXXX
	\end{center}
\fi
\IEEEpeerreviewmaketitle

\section{Introduction}

\IEEEPARstart{A}{daptive} filters~(AFs) are widely used in signal processing due to their tracking capabilities and low computational complexity. Still, they display performance trade-offs that can hinder their use in practice, such as the compromise between convergence rate or tracking and steady-state error~\cite{Sayed08a, Diniz13a}. These issues have been mitigated by modifying the filter cost functions, e.g., mixed-norm AFs, or by using variable step size~(VSS) techniques~\cite{Harris86v, Kwong92v, Mathews93s, Mayyas95r, Shin04v}. More recently, \emph{combinations of AFs} were introduced to address these trade-offs, especially in situation where the design of a single filter is intricate~\cite{Singer99u, Kozat02f, Jeronimo06m, Chambers06c, Bershad08a, Magno08i, Cassio07e, Jeronimo08a, Cassio10a, Azpicueta11a, Wilder11i, Chamon12c, Chamon12d, Vitor12l, Ferro14f}. In this approach, a pool of AFs is combined by an adaptive supervisor such that the overall system performs at least as well as the best filter in the pool, usually in the mean-square error~(MSE) sense. Typically, convex or affine supervisors are used, possibly in their normalized forms~\cite{Jeronimo06m, Bershad08a, Azpicueta08n, Candido10t}.

The most common combination structure runs each AF independently and then merges their outputs~(Fig.~\ref{F:ParallelTopology}a). We dub this structure \emph{parallel-independent}. It has been studied for different adaptive algorithms, step sizes, orders, and supervisors~\cite{Singer99u, Kozat02f, Jeronimo06m, Chambers06c, Cassio10a, Cassio07e, Bershad08a, Magno08i, Jeronimo08a}. Although effective, parallel-independent combinations display a well-known convergence stagnation regardless of the supervisor. To understand this phenomenon, consider a combination of two AFs, one fast and one slow but accurate~(Fig.~\ref{F:CyclicComponents}a). Once the faster AF has converged, the output of the combination plateaus while the slower filter does not reach a lower error level. When this happens, the combination ``switches'' filters and continues to converge.

Structural changes have been proposed to address this issue. In~\cite{Jeronimo06m}, \emph{conditional coefficients leakage} from faster to more accurate AF was introduced~(Fig.~\ref{F:ParallelTopology}b). This approach modifies the recursion of the slower AF so that its coefficient vector becomes a mixture of those from both component filters. A similar idea was explored in~\cite{Vitor12l} where the coefficients of the slower filter were conditionally replaced by those from the faster one. A different approach entirely reformulated the topology leading to \emph{incremental combinations}~\cite{Wilder11i, Chamon12d}.

A more natural solution that retains the parallel form of the combination was put forward in~\cite{Chamon12c}: \emph{cyclic coefficients feedback}~(Fig.~\ref{F:ParallelTopology}c). This structure periodically feeds back the overall coefficients of the combination to all component filters, improving their performance regardless of which AF is better at each iteration. The cyclic nature of these feedbacks is the key to exploit the output of the combination without hindering the supervisor adaptation. This topology not only addresses the convergence stagnation issue, but also improves tracking performance. In this work, we analyze this structure by

\begin{itemize}

	\item showing that VSS adaptive algorithms can be interpreted as combinations with coefficients feedback;

	\item analyzing the steady-state, tracking, and transient performance of parallel combinations of LMS filters with coefficients feedback;

	\item extending the supervisor transient model from~\cite{Vitor09t} to general activation functions;

	\item using these analyses to show that besides eliminating the stagnation issue, coefficients feedbacks increase cooperation among filters, improve tracking, and stabilize the supervising parameters;

	\item illustrating the performance of this new topology in numerical examples.

\end{itemize}

\textbf{Notation}: Lowercase boldface letters represent vectors~($\bx$) and uppercase boldface letters are used for matrices~($\bX$). Iteration indices are shown as subscripts on vectors~($\bx_i$) and in parenthesis for scalars~[$x(i)$]. We denote the steady-state value of any variable by omitting its iteration index, i.e., $x = \lim_{i\to \infty} x(i)$.

\begin{figure}[t]
	\centering
	\includesvg[width=\columnwidth]{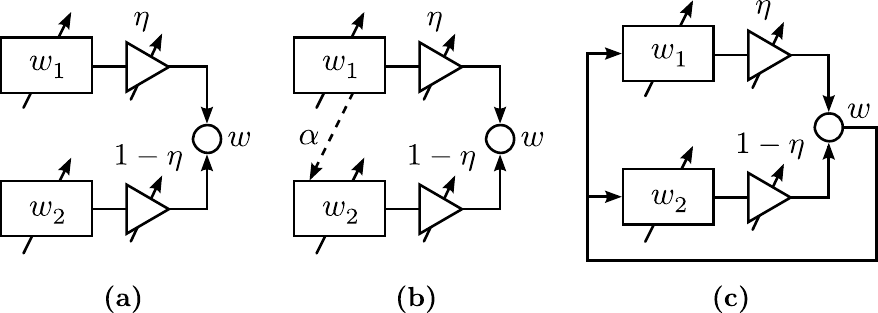}
	\caption{Parallel topologies:
				(a)~Parallel-independent;
				(b)~Parallel with transfers of coefficients---leakage or handover;
				(c)~Parallel with coefficients feedback.}
		\label{F:ParallelTopology}
\end{figure}

\section{Combinations of Adaptive Filters}
	\label{S:AFComb}

A combination of AFs is composed of three parts: the component filters, the topology, and the supervisor. \emph{Component filters} are the building blocks of combinations: they are the standalone AFs that are merged to improve their individual performances. The manner in which these AFs are merged is called the \emph{topology}. It defines how the filters interact and how the output of the combination is obtained from the output of each AF. Oftentimes this topology depends on a set of parameters that modifies the combination behavior. These \emph{supervising parameters} are adapted by the \emph{supervisor}. In the sequel, we examine each of these elements individually, starting with the component filters.

\subsection{The component filters}

In a combination of~$N$ AFs, each component filter is distinguished by using the index~$n = 1,\dots,N$. These AFs update \emph{a priori} coefficient vectors~$\bw_{n,a} \in \setR^M$ in an attempt to minimize some underlying cost function~$J_n(\bw_{n,a})$. This cost function is usually the MSE~$J_n(\bw_{n,a}) = \E \left[ d(i) - \bu_i^T \bw_{n,a} \right]^2$, where~$\bu_{i}$ is a~$M \times 1$ real-valued regressor vector and~$d(i) \in \setR$ is the desired signal. For example, in a combination of LMS filters, the recursion of the~$n$-th AF is written as
\begin{equation}\label{E:LMS}
	\bw_{n,i} = \bw_{n,a} + \mu_n \bu_{i}
		\left[ d(i) - \bu_i^{T} \bw_{n,a} \right]
		\text{,}
\end{equation}
where $\mu_n > 0$ is a step size. Using the \emph{a priori} coefficients~$\bw_{n,a}$ in~\eqref{E:LMS} we can formalize the notion of combination topology.

\begin{figure}[t]
	\centering
	\includesvg{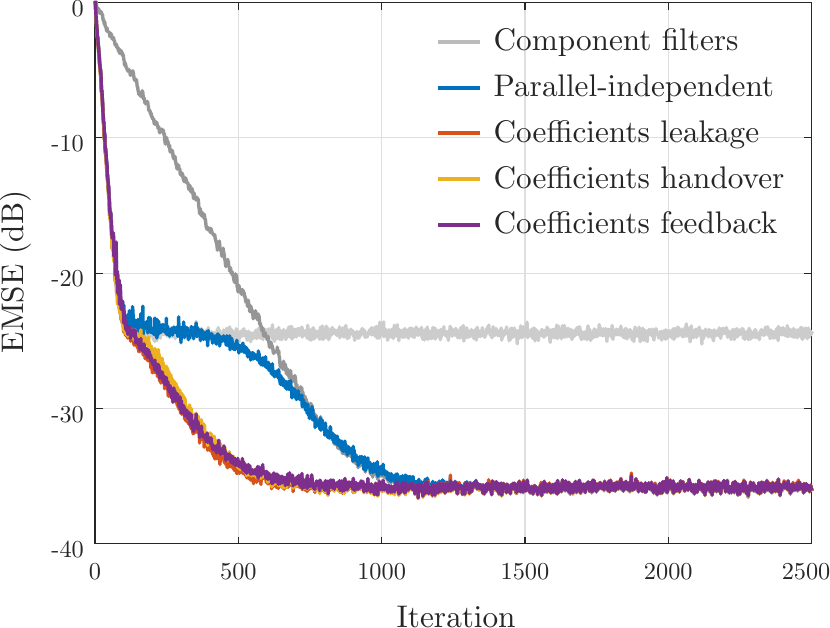}
	\caption{Effect of different topologies on the supervisor performance
	(convex supervisor). \textbf{White stationary scenario}:
		$M = 10$, $\sigma_u^2 = 1$, $\sigma_v^2 = 10^{-2}$,
		$\mu_1 = 0.05$, $\mu_2 = 0.005$;
	\textbf{Parallel-independent}: $\mu_{a} = 200$;
	\textbf{Coefficients leakage}: $\mu_{a} = 450$ and~$\alpha = 0.6$
		for~$\eta \geq 0.98$ and~$\alpha = 0$ otherwise;
	\textbf{Coefficients handover}: $L = 10$ and~$\mu_{a} = 200$;
	\textbf{Cyclic coefficients feedback}: $L = 90$ and~$\mu_a = 270$.}
		\label{F:TopologySupervisor}
\end{figure}

\begin{figure*}[t]
	\centering
	\includesvg{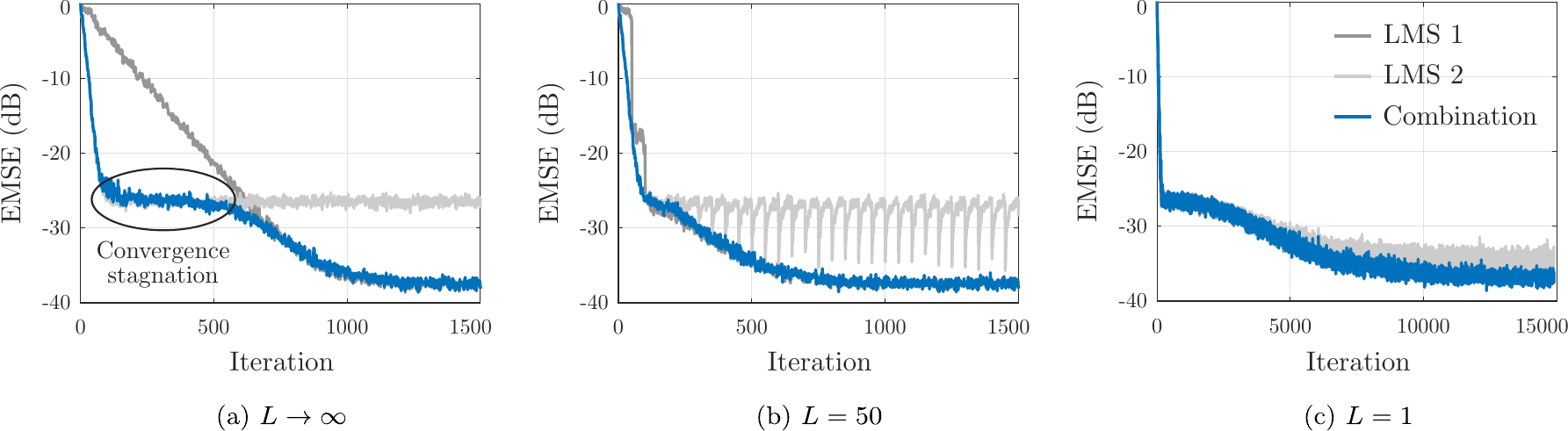}
	\vspace*{-10pt}
	\caption{Combination of LMS filters with cyclic coefficients feedback. 	\textbf{White stationary scenario}~(see Section~\ref{S:Sims}):
		$M = 7$, $\sigma_u^2 = 1$, $\sigma_v^2 = 10^{-2}$, 
		$\mu_1 = 0.05$, and $\mu_2 = 0.005$.
	(a)~\textbf{Parallel-independent}~($L \to \infty$):
		$\mu_a = 200$~(convex supervisor);
	(b)~\textbf{Cyclic coefficients feedback}: $L = 50$
		and~$\mu_a = 200$~(convex supervisor);
	(c)~\textbf{Coefficients feedback}: $L = 1$
		and~$\mu_a = 300$~(convex supervisor).}
		\label{F:CyclicComponents}
\end{figure*}

\subsection{The topology}
	\label{S:topology}

Notice from~\eqref{E:LMS} that AFs have 3 inputs: the regressor~$\bu_i$, the desired signal~$d(i)$, and the \emph{a priori} coefficient vector~$\bw_{n,a}$. Filters in a combination interact using the latter. In fact, the topology of a combination is formally defined by specifying the~$\bw_{n,a}$ and a function that maps the component filters to the output of the combination. In this work, we focus on parallel topologies, i.e., those for which the update~\eqref{E:LMS} can be evaluated for all~$n$ simultaneously. Formally, a parallel topology is one in which the~$\{ \bw_{n,a} \}$ at iteration~$i$ are functions only of the~$\{ \bw_{n,j} \}$ for~$j < i$. This is in contrast to incremental or series topologies, in which the update of a component filter may depend on the state of another AF at the current iteration~\cite{Wilder11i}. We define the output of parallel combinations as
\begin{equation}\label{E:Comb}
	\bw_{i} = \sum_{n=1}^{N} \eta_{n}(i) \bw_{n,i}
		\text{,}
\end{equation}
where~$\bw_{i}$ is the \emph{global coefficient vector} and~$\{ \eta_n(i) \}$ are the supervising parameters.

The most common parallel topology is the \emph{parallel-independent}, where~$\bw_{n,a} = \bw_{n,i-1}$~(Fig.~\ref{F:ParallelTopology}a). In other words, the component filters run independently and their cooperation arises solely at the output of the combination~\eqref{E:Comb}~\cite{Jeronimo06m, Bershad08a, Magno08i, Vitor09t, Kozat10u, Kozat11t}. This structure suffers from a well-known convergence stagnation issue illustrated in Fig.~\ref{F:CyclicComponents}a. Notice that the overall combination appears to stall as the faster filter approaches steady-state until the slower filter output error catches up. Although this effect is more prominent in stationary scenarios, it occurs for all common supervisors and suggests that structural changes are required to overcome it.

In~\cite{Jeronimo06m}, convergence stagnation was addressed using a conditional \emph{coefficients leakage} scheme~(Fig.~\ref{F:ParallelTopology}b). Assuming filter~$2$ is slower and more accurate, this topology adapts filter~$1$ independently, i.e., $\bw_{1,a} = \bw_{1,i-1}$, and changes the second filter recursion to
\begin{align}\label{E:CoefLeak}
	\bw_{2,i} &= \alpha \bw_{1,i-1} +
		(1 - \alpha) [\bw_{2,i-1} + \mu_{2} \bu_{i} e_{2}(i)]
		\text{,}
\end{align}
where~$e_2(i) = d(i) - \bu_i^{T} \bw_{2,i-1}$, $\alpha = \bar{\alpha}\, \mathbb{I}_{\eta(i) > \eta_\text{th}}$, for constants~$\bar{\alpha} \in [0,1]$ and~$\eta_\text{th} \approx 1$, and~$\mathbb{I}$ denotes the indicator function. It is straightforward to see why the threshold test in the indicator function is paramount: unless~$\alpha$ vanishes, filter~$1$ will perturb the steady-state of filter~$2$ and the overall combination will eventually be worse than the second filter alone.

An alternative transfer of coefficients was proposed~\cite{Vitor12l}. \emph{Coefficients handover} once again assumes that the component filter~$2$ is slower and more accurate and cyclically and conditionally assigns~$\bw_{2,i} = \bw_{1,i}$. Although it uses the same condition as in the coefficients leakage scheme, this solution is more effective in addressing the convergence stagnation issue. Moreover, it can also be seen as a complexity reduction technique since the second component is not updated when handover occurs~\cite{Vitor12l}.

Though effective~(see Fig.~\ref{F:TopologySupervisor}), these methods have downsides from an application point of view. In particular, they are based on unidirectional transfers, i.e., from filter~$1$ to filter~$2$. It may however be hard to guarantee in practice that filter~$1$ will always be faster. This is the case, for instance, in tracking applications or varying signal-to-noise ratio~(SNR). This issue is further accentuated in combinations with~$N > 2$ AFs in which there are~$N(N-1)$ possible transfers to choose from. Moreover, these techniques regularly require conditional tests to check whether transfers should take place, which places an additional burden on the implementation.

Before we show how a topology with coefficients feedback addresses these issues, we explain the final element of combinations, the supervisor.

\subsection{The supervisor}
	\label{S:SupervisingRules}

The supervisor is responsible for adjusting the supervising parameters of the topology, playing a fundamental role in the combination performance. The most common approach is to adapt~$\eta_n(i)$ in~\eqref{E:Comb} so as to minimize the global MSE~$\E e(i)^{2}$ under an unbiasedness constraint, where~$e(i) = d(i) - \bu_i \bw_{i-1}$ is the \emph{global output estimation error}. For parallel topologies, this constraint can be expressed as~\cite{Jeronimo06m, Bershad08a, Kozat10u}
\begin{equation}\label{E:Unbiasedness}
	\sum_{n=1}^{N} \eta_{n}(i) = 1
		\text{,} \quad \text{for all } i \text{.}
\end{equation}
Without loss of generality, the remain of this paper considers the~$N = 2$ case, given that larger combinations can be built from smaller ones using hierarchical combination techniques~\cite{Cassio07e, Jeronimo08a}. Using~\eqref{E:Unbiasedness}, the global coefficient vector in~\eqref{E:Comb} therefore reduces to
\begin{equation}\label{E:2Comb}
	\bw_{i} = \eta(i) \, \bw_{1,i} + \left[ 1-\eta(i) \right] \bw_{2,i}
		\text{.}
\end{equation}

The global MSE minimization is typically carried out using stochastic gradient descent, mirroring the adaptation of standalone AFs such as~\eqref{E:LMS}. Additionally, a strictly increasing activation function~$f$ can be used to reduce the variance of the supervising parameters, in which case the supervisor adapts an internal state~$a$ that determines~$\eta(i)$. This general supervising parameter adaptation procedure can be described as
\begin{subequations}\label{E:SupervisorModel}
\begin{align}
	a(i) &= a(i-1) + \mu_{a} e(i) \left[ y_{1}(i) - y_{2}(i) \right]
		f^{\prime}[a(i-1)]
		\label{E:adaptA}
	\\
	\eta(i) &= f\left[ a(i) \right]
		\label{E:etaModel}
		\text{,}
\end{align}
\end{subequations}
where~$y_n = \bu_i^T \bw_{n,i-1}$ is the output of the $n$-th component, $\mu_a > 0$ is a step size, and~$f^{\prime}$ denotes the derivative of~$f$~\cite{Cassio07e, Cassio10a}. The auxiliary variable~$a$ is usually restricted to some range~$[a^{-},a^{+}]$ by saturating~\eqref{E:adaptA}. Note that the restrictions on~$f$ imply that~$f^{\prime} > 0$ and guarantees that~\eqref{E:adaptA} descends along the gradient and avoids stalling the adaptation algorithm. In other words, it ensures that there is no~$a$ for which~$f^{\prime}(a)$ vanishes and the supervising parameter becomes independent of the component filters.

For appropriate choices of the activation function in~\eqref{E:SupervisorModel} we can recover the two most widely used supervising algorithms. The \emph{affine supervisor} is obtained by letting~$f(a) = a$ in~\eqref{E:SupervisorModel} and reads
\begin{equation}\label{E:AffineSupervisor}
	\eta(i) = \eta(i-1) + \mu_{\eta} e(i)
		\left[ y_{1}(i) - y_{2}(i) \right]
		\text{.}
\end{equation}
The value of~$\eta(i)$ is typically constrained to improve steady-state stability by, for instance, taking~$\eta(i) \in [-0.2,1.2]$~\cite{Bershad08a, Jeronimo16c}. The \emph{convex supervisor} takes~$f$ to be a sigmoidal function, so that~\eqref{E:SupervisorModel} becomes
\begin{subequations}\label{E:ConvexSupervisor}
\begin{align}
	a(i) &= a(i-1) + \mu_{a} e(i) \left[ y_{1}(i) - y_{2}(i) \right]
		\times
	\notag\\&\qquad\qquad\qquad\qquad\qquad
		\eta(i-1) \left[ 1 - \eta(i-1) \right]
		\label{E:ConvexSupervisor_a}
	\\
	\eta(i) &= \dfrac{1}{1 + e^{-a(i)}}
		\label{E:ConvexSupervisor_f}
\end{align}
\end{subequations}
To prevent stalling, $a(i)$ is usually restricted to~$[-4,4]$ either by saturating~$a$ or modifying~\eqref{E:ConvexSupervisor_f}~\cite{Jeronimo06m, Gredilla10a}.

Normalized versions of supervisors~\eqref{E:AffineSupervisor} and~\eqref{E:ConvexSupervisor} are obtained by taking~$\mu_{\eta},\mu_a = \tilde{\mu} / p(i)$ with~$p(i) = \beta e(i)^2 + (1-\beta) \, p(i-1)$ for some normalized step size~$\tilde{\mu} > 0$~\cite{Azpicueta08n, Candido10t, Jeronimo16c}. Although we use the unnormalized~\eqref{E:SupervisorModel} in our analyses, we illustrate the fact that the coefficients feedback topology is effective regardless of the supervisor by also displaying results for the normalized versions.

\begin{figure}[t]
	\centering
	\includesvg{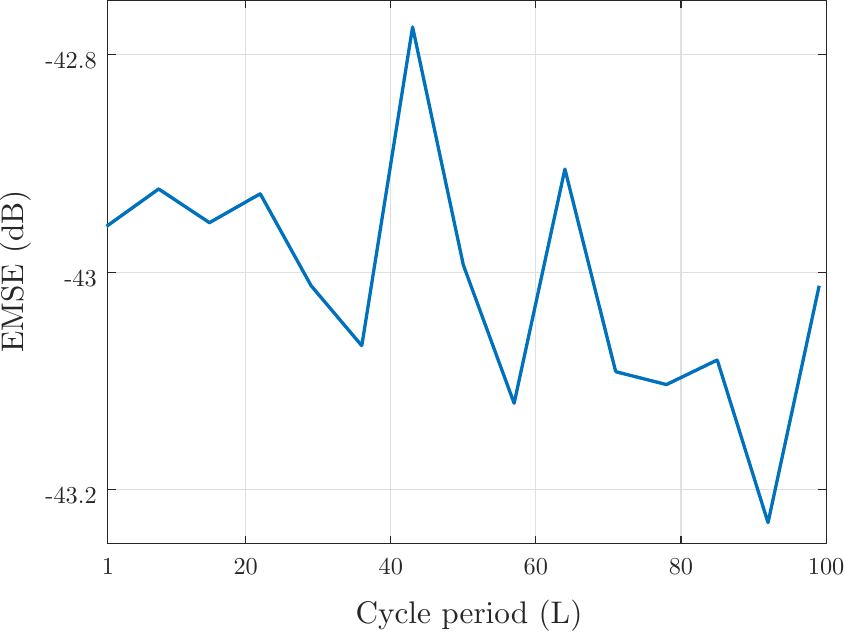}
	\caption{Steady-state performance of combinations with coefficients feedback for different~$L$. \textbf{Correlated scenario}:
		$M = 5$, $\sigma_u^2 = 1$, $\sigma_v^2 = 10^{-2}$, $\gamma = 0.7$,
			$\mu_1 = 0.01$, and~$\mu_2 = 0.002$.
			\textbf{Convex supervisor}: $\mu_{a} = 100$.
		Steady-state value is an average of 1000 iterations~(see Section~\ref{S:Sims}).}
		\label{F:cycleLengthStat1}
\end{figure}

\section{Cyclic Coefficients Feedback}
	\label{S:CoefFB}

Regardless of using convex or affine supervisors, normalized or not, convergence stagnation remains an issue in parallel-independent combinations. To mitigate this issue, a new topology was introduced in~\cite{Chamon12c} as an alternative to the coefficients transfers from Section~\ref{S:topology}. It is based on feeding back the globl coefficients to all component filters at every iteration~(Fig.~\ref{F:ParallelTopology}c). Formally, the \emph{a priori} coefficients in~\eqref{E:LMS} are defined as~$\bw_{n,a} = \bw_{i-1}$, for~$\bw_i$ as in~\eqref{E:Comb}. This is motivated by the fact that the supervisor adapts~\eqref{E:Comb} so that the global coefficients minimize the overall MSE. Hence, these feedbacks enforce that each AF in the combination updates the best available coefficient vector estimate at each iteration.

In practice, however, coefficients feedbacks have the effect of slowing down the supervisor adaptation, which degrades the combination performance~(Fig.~\ref{F:CyclicComponents}c). This is due to the fact that feedbacks reduces the difference between the component filters outputs~$y_n(i)$, since they all update the same \emph{a priori} coefficients. It is then clear from~\eqref{E:SupervisorModel} that the~$\eta_n(i)$ will change slowly. A simple solution to this issue is to make the feedback cyclical by taking
\begin{equation}\label{E:Feedback}
	\bw_{n,a} = \delta(i - rL) \bw_{i-1} +
		\left[ 1 - \delta(i - rL) \right] \bw_{n,i-1}
		\text{,}
\end{equation}
where~$L$ is a constant \emph{cycle period}, $\delta(i)$ is the Kronecker delta, and~$r \in \setN$~(Fig.~\ref{F:CyclicComponents}b).

Combinations with cyclic coefficients feedback have two limiting behaviors: (i)~for~$L = 1$, the relation in~\eqref{E:Feedback} reduces to~$\bw_{n,a} = \bw_{i-1}$ and the global coefficients are always fed back to the component filters; (ii)~for~$L \to \infty$, we can write~\eqref{E:Feedback} as~$\bw_{n,a} = \bw_{n,i-1}$ and the usual parallel-independent combination from Section~\ref{S:topology} is recovered. However, the advantage of coefficients feedback appear when~$1 < L \ll \infty$. It is worth noting that the overall performance of this topology is robust for a wide range of~$L$~(see Figs.~\ref{F:cycleLengthStat1} and~\ref{F:cycleLengthStat2}), so that values between~$50$ and~$150$ typically work well. Alternatively, a procedure to design the cycle period was proposed in~\cite{Chamon12c}.

Coefficients feedback provides all component filters with the best coefficients estimate available to the combination, i.e., the global coefficients. Moreover, they are neither directional nor limited to any pair of component filters as in the structures from Section~\ref{S:topology}. It is therefore straightforward to extend beyond two filter combinations. In terms of performance, this topology not only effectively addresses convergence stagnation, but can also improve the combination misadjustment, especially in nonstationary scenarios. These performance improvements are examined in the next two sections by analyzing the steady-state, tracking, and transient behavior of a combination of LMS filters with cyclic coefficients feedback. Although any AFs can be used in these combinations, we focus on LMS filters for the purpose of analysis. Before proceeding, however, we discuss a conceptual consequence of the coefficients feedback topology by formalizing the relation between combinations of AFs and VSS algorithms.

\begin{figure}[t]
	\centering
	\includesvg{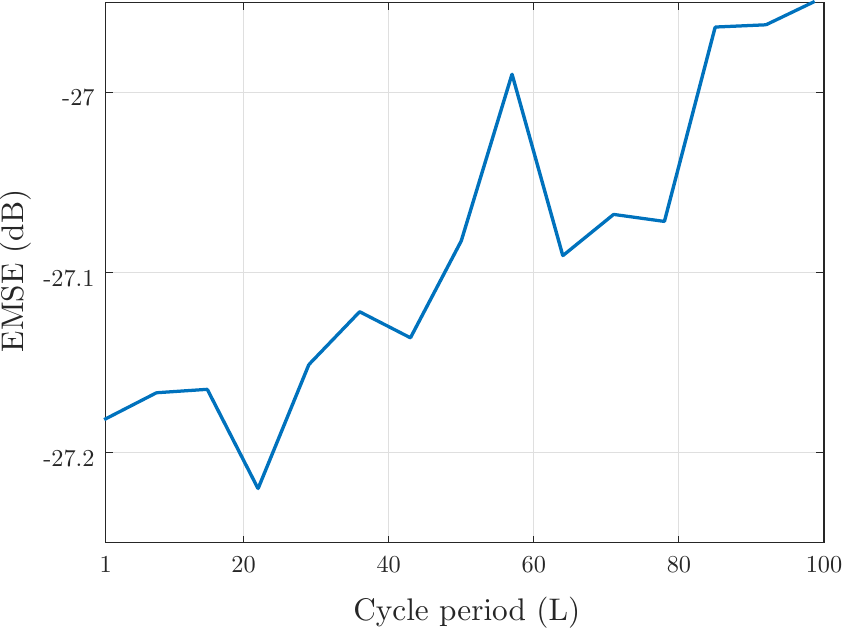}
	\caption{Tracking performance of combinations with coefficients feedback for different~$L$. \textbf{White nonstationary scenario}:
		$M = 5$, $\sigma_u^2 = 1$, $\sigma_v^2 = 10^{-2}$,
		$\bQ = 10^{-5} \bI$, $\mu_1 = 0.07$, and~$\mu_2 = 0.01$.
		\textbf{Convex supervisor}: $\mu_a = 100$.
		Steady-state value is an average of 1000 iterations~(see Section~\ref{S:Sims}).}
		\label{F:cycleLengthStat2}
\end{figure}

\subsection{VSS algorithms as combinations of adaptive filters}
	\label{S:vss}

The close relation between standalone VSS AFs and combinations of AFs was already observed in~\cite{Jeronimo06m}. The parallel-independent topology, however, cannot be used as a step size adaptation structure. This barrier prevents, for instance, the effective use of these combinations in nonstationary scenarios. In fact, a parallel-independent combination of LMS filters can only match the tracking performance of a standalone LMS filter with an optimally designed step size~$\mu^o$~\cite[Lemma~7.5.1]{Sayed08a} if one of its component filters has step size~$\mu^o$~\cite{Vitor10t}.

The coefficients feedback topology bridges this conceptual and practical gap. To see how this is the case, let all AFs of a combination have update equations of the form
\begin{equation*}
	\bw_{n,i} = \bw_{n,a} + \mu_n \bp_i(\bw_{n,a})
		\text{,} \quad \text{for } n = 1,\dots,N
\end{equation*}
where~$\bp_i$ denotes the update direction at iteration~$i$. For example, the LMS filter in~\eqref{E:LMS} is recovered by taking~$\bp_{i}(\bw_{n,a}) = \bu_{i} \left[ d(i) - \bu_i^T \bw_{n,a} \right]$. When feedback occurs at every iteration, i.e., $L = 1$ in~\eqref{E:Feedback}, the update reads
\begin{equation*}
	\bw_{n,i} = \bw_{i-1} + \mu_n \bp(\bw_{i-1})
\end{equation*}
for all component filters. It is then ready from~\eqref{E:Comb} and~\eqref{E:Unbiasedness} that the output of an unbiased combination is given by
\begin{equation}\label{E:VSS}
	\bw_{i} = \bw_{i-1} + \hat{\mu}(i) \bp( \bw_{i-1} )
		\text{,}
\end{equation}
where~$\hat{\mu}(i) = \sum \eta_n(i) \mu_n$ is an iteration dependent combination of the component filters step sizes~$\{ \mu_n \}$. Hence, the output of the combination has the form of a VSS version of its component filters.

Beyond its conceptual value, \eqref{E:VSS} also provides a viable alternative to step size adaptation rules such as~\cite{Kwong92v, Harris86v, Mayyas95r, Mathews93s}. Numerical experiments using combinations of LMS filters with cyclic coefficients feedback show that the resulting VSS algorithms are robust to changes in the environment and can track the optimal step-size even in stringent nonstationary scenarios~(e.g., see Fig.~\ref{F:2LMSNonStat}).

It is worth noting that step size adaptation is only one possible application of coefficients feedback. In fact, these combinations are only related to VSS algorithms when all the AFs in the pool use the same adaptation algorithm and feedback occurs at every iteration. Combinations with cyclic coefficients feedback are therefore more general structures.

\begin{figure}[t]
	\centering
	\includesvg{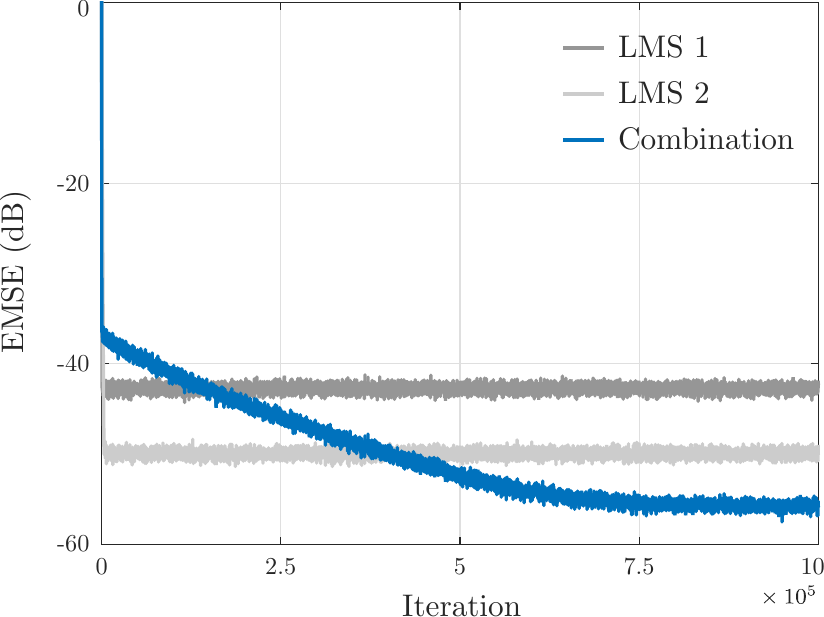}
	\caption{Steady-state of combination of LMS filters with cyclic coefficients feedback: EMSE performance.
	\textbf{White stationary scenario}~(see Section~\ref{S:Sims}):
		$M = 10$, $\sigma_u^2 = 1$, $\sigma_v^2 = 10^{-3}$, 
		$\mu_1 = 0.01$, $\mu_2 = 0.002$, and~$L = 20$.
	\textbf{Affine supervisor}: $\mu_\eta = 1.5$
		and~$\eta \geq -0.25 \Rightarrow \mubar \geq 0$.}
		\label{F:StatAffine1}
\end{figure}

\section{Steady-State and Tracking Performance}
	\label{S:SSTracking}

We begin the analysis of combinations of LMS filters with coefficients feedback by studying their steady-state performance in a system identification scenario. Although AFs can be used in a myriad of setups, their analyses are typically carried out in this setting as it is representative of applications such as echo cancellation, time delay estimation, and adaptive control~\cite{Sayed08a, Diniz13a, Huang06a, So01c, Widrow08a}. Let~$\bu_{i}$ be a~$M \times 1$ real-valued regressor vector with covariance matrix~$\bR_u = \E \bu_i \bu_i^T$ and let~$d(i)$ be a scalar measurement of the form
\begin{equation}\label{E:DataModel}
	d(i) = \bu_{i}^T \bw^{o}_i + v(i)
		\text{,}
\end{equation}
where~$\bw^{o}_i \in \setR^M$ represents the unknown system at iteration~$i$ and~$v(i)$ is a zero-mean white Gaussian noise with variance~$\sigma_{v}^{2}$. The system coefficients~$\bw^{o}_i$ follow the first-order Markov process
\begin{equation}\label{E:RandomWalkModel}
	\bw^o_{i} = \bw^o_{i-1} + \bq_i
		\text{,}
\end{equation}
where~$\bw^o_{-1} = \bw^o$ is the initial condition and~$\bq_i \in \setR^M$ is a zero-mean i.i.d{.} sequence of vectors with covariance matrix~$\bQ = \E \bq_i \bq_i^T$. The stationary case is recovered for\ $\bq_i = \bzero \Rightarrow \bQ = \bzero$. To make the derivations more tractable, we adopt the following typical assumptions~\cite{Sayed08a, Diniz13a}:

\begin{assumption}[Noise independence]\label{A:NoiseIndependence}

$\{ u(i), v(j) \}$ are independent for all $i,j$.

\end{assumption}

\begin{assumption}[Random walk independence]\label{A:RandomWalk}

The random variable $\bq_i$ is statistically independent of the initial conditions $\{ \bw^o_{-1}, \bw_{n,-1} \}$, of $\{ \bu_j, v(j) \}$ for all~$i,j$, and of~$\{ d(j) \}$ for~$j < i$.

\end{assumption}

Our performance metric of interest is the excess MSE~(EMSE) defined as follows. Let~$\tilw_{i} = \bw^o_i - \bw_i$ and~$\tilw_{n,i} = \bw^o_i - \bw_{n,i}$ be the global and local coefficient error vectors respectively. Define the global \emph{a priori} error as~$e_a(i) = \bu_i^T (\bw^o_i - \bw_{i-1})$ and its local counterpart as~$e_{a,n}(i) = \bu_i^T (\bw^o_i - \bw_{n,i-1})$. For a combination of two filters, we then have
\begin{equation}\label{E:EMSEs}
\begin{aligned}
	\zeta(i) &= \E e_a^2(i)
				& \text{(global EMSE)}
	\\
	\zeta_n(i) &= \E e_{a,n}^2(i) \text{, for } n = 1,2
				& \text{(local EMSE)}
	\\
	\zeta_{12}(i) &= \E e_{a,1}(i) e_{a,2}(i)
				& \text{(cross-EMSE)}
\end{aligned}
\end{equation}

\begin{figure}[t]
	\centering
	\includesvg{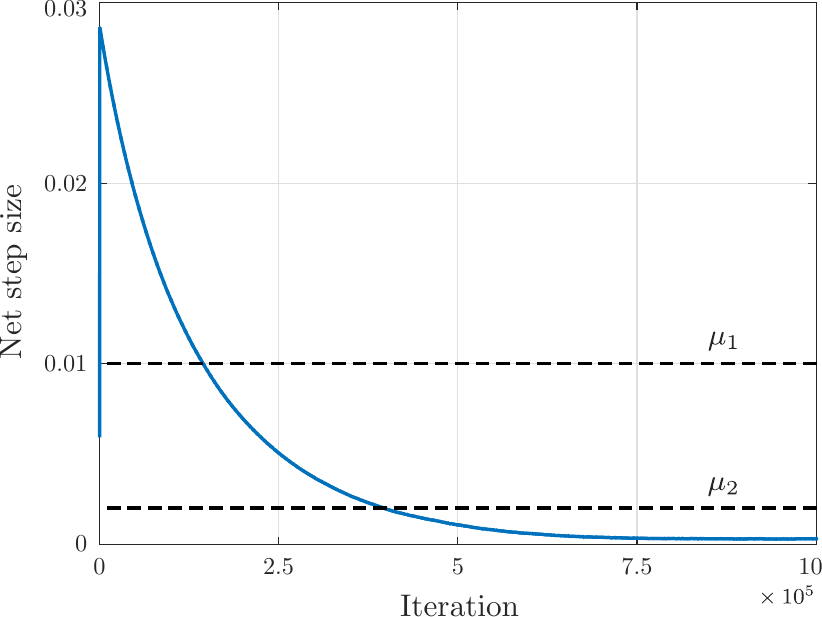}
	\caption{Steady-state of combination of LMS filters with cyclic coefficients feedback: equivalent step-size of combination~(same setting as Fig.~\ref{F:StatAffine1}).}
		\label{F:StatAffine2}
\end{figure}

Before proceeding, we derive some relations that will be useful throughout our analysis. First, subtract the global coefficients~\eqref{E:2Comb} from~$\bw^o_i$ to get
\begin{equation}\label{E:GlobalCoefError}
	\tilw_{i} = \eta(i) \tilw_{1,i}
		+ \left[ 1-\eta(i) \right] \tilw_{2,i}
		\text{.}
\end{equation}
Also, subtracting~$\bw^o_i$ from~\eqref{E:2Comb} at~$i-1$ and multiplying by~$\bu_i^T$ yields
\begin{equation}\label{E:GlobalAPriori}
	e_a(i) = \eta(i-1) e_{a,1}(i) + \left[ 1-\eta(i-1) \right] e_{a,2}(i)
		\text{.}
\end{equation}
Using the definition of the component filters output, we obtain
\begin{equation}\label{E:outputDiff}
	y_{1}(i) - y_{2}(i) = \bu_i^T \bw_{1,i-1} - \bu_i^T \bw_{2,i-1}
		= e_{a,2}(i) - e_{a,1}(i)
		\text{,}
\end{equation}
by adding and subtracting~$\bu_i^T \bw^o_i$. Finally, we can use~\eqref{E:DataModel} to write the global estimation error~$e(i)$ in terms of the \emph{a priori} error as
\begin{equation}\label{E:estimationAPrioriError}
\begin{aligned}
	e(i) &= d(i) - \bu_i^T \bw_{i-1}
	\\
	{}&= \bu_{i}^T \bw^{o}_i + v(i) - \bu_i^T \bw_{i-1} = e_a(i) + v(i)
		\text{.}
\end{aligned}
\end{equation}
Note from~\eqref{E:estimationAPrioriError} that under~A.\ref{A:NoiseIndependence} the global MSE is given by~$\E e(i)^2 = \zeta(i) + \sigma_v^2$. Hence, although we carry out all derivations for the EMSE, our results can readily be used to measure performance with respect to the MSE.

We are now ready to state the main result of this section.

\begin{theorem}\label{T:SS}

Consider the system identification scenario in~\eqref{E:DataModel}--\eqref{E:RandomWalkModel} and a supervisor of the form~\eqref{E:SupervisorModel}. Under assumptions~A.\ref{A:NoiseIndependence}--A.\ref{A:SSSupervisorVar}, the steady-state value of the global EMSE of a combination of LMS filters with coefficients feedback is
\begin{equation}\label{E:2LMSFB}
	\zeta =
	\frac{
		\mubar \trace(\bR_u) \sigma_v^2 + \mubar^{-1} \trace (\bQ)
	}{
		2 - \mubar \trace(\bR_u)
	}
		\text{,}
\end{equation}
where~$\mubar = \etabar \mu_1 + (1-\etabar) \mu_2$, $\etabar$ is the restriction of~$\eta^o$ onto~$[f(a^{-}),f(a^{+})]$, and
\begin{equation}\label{E:SSSupervisor}
	\eta^o = \frac{
		c - \trace(\bR_u) \left[ \trace(\bQ) + 2 \mu_2 \sigma_v^2 \right]
	}{
		2 (\mu_1 - \mu_2) \trace(\bR_u) \sigma_v^2
	}
		\text{,}
\end{equation}
with~$c^2 = \trace(\bR_u)^2 \trace(\bQ)^2 + 4 \sigma_v^2 \trace(\bR_u) \trace(\bQ)$ is the supervising parameter that minimizes the global MSE.

\end{theorem}

Given the relation between combinations with coefficients feedback and VSS algorithms, it is not surprising that~\eqref{E:2LMSFB} is equivalent to the expression for the steady-state of a standalone LMS filter with step size~$\mubar$~\cite[Lemma~7.5.1]{Sayed08a}. Theorem~\ref{T:SS}, however, also states that~$\mubar$ is chosen by the supervisor to minimize the global MSE~(see Proposition~\ref{T:etabar}). In the nonstationary case, coefficients feedback therefore force the supervisor to explicitly track the optimal step size~$\mu^o$ from~\cite[Lemma~7.5.1]{Sayed08a}. For convex or affine supervisors, it then suffices that~$\mu_1 \leq \mu^o \leq \mu_2$ for the coefficients feedback topology to outperform any parallel-independent combination that does not have at least one component filter with step size~$\mu^o$. Even if it does, recall that~$\mu^o$ depends on~$\bR_u$, $\bQ$, and~$\sigma_v^2$, so that parallel-independent combinations would not robust to changes in the environment. The adaptive step size effect of coefficients feedback, however, mitigates this issue.

Theorem~\ref{T:SS} may no longer be valid in stationary scenarios, especially if the range of the activation function~$f$ from~\eqref{E:SupervisorModel} is unbounded, e.g., for the affine supervisor~\eqref{E:AffineSupervisor}. In this case, $\bQ = \bzero$ implies that~$\eta^o \to \mu_2 (\mu_1 - \mu_2)^{-1}$ and~$\mubar \to 0$ at steady-state. However, Figs.~\ref{F:StatAffine1} and~\ref{F:StatAffine2} illustrates that, although~$\mubar$ and the EMSE decrease considerably, they both converge to a non-vanishing steady-state. This occurs because as the EMSE diminishes, the variance of the supervising parameter becomes non-negligible, violating the approximations used to derive Theorem~\ref{T:SS}. Nevertheless, Theorem~\ref{T:SS} holds for the convex supervisor or any other case in which~$[f(a^{-}),f(a^{+})]$ does not allow~$\mubar$ to vanish. Otherwise, we can resort to the transient analysis in Section~\ref{S:Transient}.

The remainder of this section is dedicated to proving Theorem~\ref{T:SS}. We proceed in three steps. First, we introduce the different steady-state regimes of the cyclic coefficients feedback topology and argue that for small cycle periods the combination performs as if~$L = 1$~(Section~\ref{S:SSGlobal}). Then, we derive the component filters error statistics required to obtain~\eqref{E:2LMSFB}~(Section~\ref{S:SSComponents}) and determine the steady-state value of the supervising parameter~(Section~\ref{S:SSSupervisor}). Recall that even though all derivations are carried out for two component filters, the results can be extended to arbitrary~$N$ by evaluating additional error statistics or using hierarchical combinations~\cite{Cassio07e, Jeronimo08a}.

\subsection{Global steady-state analysis}
	\label{S:SSGlobal}

Analyzing the asymptotic performance of combinations with cyclic coefficients feedback is intricate due to the dynamic nature of their steady-state. Indeed, for cycle periods~$1 < L < \infty$, their topology alternates between parallel-independent and parallel with coefficients feedback, so that the component filters and the output of the combination reach a \emph{cyclostationary} regime as~$i \to \infty$, as illustrated in Fig.~\ref{F:CyclicComponents}b. In contrast, the limit cases~$L \to \infty$ and~$L = 1$ have static topologies and stationary steady-state errors~(Figs.~\ref{F:CyclicComponents}a and~\ref{F:CyclicComponents}c respectively). To account for cyclic feedback without resorting to a full transient analysis as in Section~\ref{S:Transient}, we leverage the fact that since feedbacks are cyclical solely to avoid stalling the supervisor adaptation, we care only about performance for small~$L$. Hence, we can use the following approximation, is valid for a significant range of~$L$~(Figs.~\ref{F:cycleLengthStat1} and~\ref{F:cycleLengthStat2}):

\begin{assumption}[Small cycle period]\label{A:SSCyclePeriod}

For small cycle period, the combination performs at steady-state as if~$L = 1$.

\end{assumption}

Moreover, we can decouple the analysis of the global error and the supervisor by assuming that

\begin{assumption}[Steady-state supervisor separation principle]\label{A:SSSupervisorSeparation}

At steady-state, $\eta(i)$ varies slowly compared to the coefficients error vector $\tilw_{n,i}$ and the \emph{a priori} errors $e_{a,n}(i)$, so that their expected values can be separated as in $\E [ \eta(i) \tilw_{n,i} ] = \E \eta(i) \E \tilw_{n,i}$ and $\E [ \eta(i) e_{a,n}(i) ] = \E \eta(i) \E e_{a,n}(i)$, $i \to \infty$.%
%

\end{assumption}

\noindent This assumption is supported by simulations and has been successfully used in the steady-state analysis of parallel-independent topologies, i.e., $L \to \infty$, for both affine~\cite{Bershad08a, Candido10t, Kozat11t} and convex~\cite{Jeronimo06m, Magno08i} supervisors. Using~A.\ref{A:SSSupervisorSeparation}, we can now derive the result of Theorem~\ref{T:SS} in two parts, by first analyzing the component filters and then the supervisor.

\subsection{Component filters steady-state analysis}
	\label{S:SSComponents}

Start by taking the expected value of the squared norm of the coefficient error vector~\eqref{E:GlobalCoefError} to get
\begin{align*}
	\E \norm{\tilw_i}^2 &= \E \eta^2(i) \norm{\tilw_{1,i}}^2
	+ 2 \E \eta(i) [1-\eta(i)] \tilw_{1,i}^T \tilw_{2,i}
	\\
	{}&+ \E [1-\eta(i)]^2 \norm{\tilw_{2,i}}^2
		\text{.}
\end{align*}
Then, using A.\ref{A:SSSupervisorSeparation}, we can separate the expectations and obtain
\begin{equation}\label{E:GlobalMSD}
\begin{aligned}
	\E \norm{\tilw_i}^2 &= \E\eta^2 \E \norm{\tilw_{1,i}}^2
		+ 2 \E \eta [1-\eta] \E \tilw_{1,i}^T \tilw_{2,i}
	\\
	{}&+ \E[1-\eta]^2 \E \norm{\tilw_{2,i}}^2
		\text{,} \quad \text{as } i \to \infty
		\text{.}
\end{aligned}
\end{equation}
The component filters error statistics required in~\eqref{E:GlobalMSD} can be evaluated using an energy conservation argument as in~\cite{Sayed08a}. We do so by adopting the following assumption to derive the variance and covariance relations of the AFs in the combination.

\begin{assumption}[Data separation principle]\label{A:SSData}

$\norm{\bu_i}^2$ is independent of $e_{a,n}(i)$, and consequently $e_a(i)$, at steady-state.

\end{assumption}

\noindent We then have the following proposition:

\begin{proposition}\label{T:VarCovar}

Under~A.\ref{A:NoiseIndependence}, A.\ref{A:RandomWalk}, and~A.\ref{A:SSData}, the variance and covariance relations for the LMS component filters of a combination with coefficients feedback~($L = 1$) are given by
\begin{subequations}\label{E:PreFB}
\begin{align}
	\E \norm{\tilw_{n,i}}^2 &= \E \norm{\tilw_{i-1}}^2 + \trace(\bQ)
	\notag\\
	{}&- 2 \mu_{n} \zeta(i)
		+ \mu_{n}^2 \trace(\bR_u) \left[ \zeta(i) + \sigma_v^2 \right]
		\text{,}
	\\
	\E \tilw_{1,i}^T \tilw_{2,i} &= \E \norm{\tilw_{i-1}}^2 + \trace(\bQ)
	- (\mu_1 + \mu_2) \zeta(i)
	\notag\\
	{}&+ \mu_1 \mu_2 \trace(\bR_u) \left[ \zeta(i) + \sigma_v^2 \right]
		\text{,}
\end{align}
\end{subequations}
as~$i \to \infty$.
\end{proposition}

\begin{proof}
See Appendix~\ref{AP:VarCovar}.
\end{proof}

Note that, as opposed to the analysis of parallel-independent combinations~\cite{Jeronimo06m, Bershad08a, Candido10t}, the recursions in~\eqref{E:PreFB} are coupled. Substituting~\eqref{E:PreFB} into the steady-state relation~\eqref{E:GlobalMSD} gives
\begin{equation}\label{E:PreFB2}
\begin{aligned}
	\E \norm{\tilw_{i}}^2 &= \E \norm{\tilw_{i-1}}^2 + \trace(\bQ)
	- 2 \E \hat{\mu} \zeta(i)
	\\
	{}&+ \E \hat{\mu}^2
		\trace(\bR_u) \left[ \zeta(i) + \sigma_v^2 \right]
		\text{, as } i \to \infty \text{,}
\end{aligned}
\end{equation}
where~$\hat{\mu} = \eta \mu_{1} + [1-\eta] \mu_{2}$ is the steady-state \emph{net step size} of the combination with~$\eta = \lim_{i \to \infty} \eta(i)$. Note that~$\hat{\mu}$ is a random variable due to its dependence on~$\eta$.

To conclude the component analysis, recall that under~A.\ref{A:SSCyclePeriod} we only consider the~$L = 1$ case in which the combination topology is static. Hence,~$\E \norm{\tilw_{i}}^2$ eventually converges to a stationary value for a proper choice of~$\mu_1$ and~$\mu_2$, and~$\E \norm{\tilw_{i}}^2 = \E \norm{\tilw_{i-1}}^2$ as~$i \to \infty$. We can therefore rewrite~\eqref{E:PreFB2} as
\begin{equation}\label{E:2LMSFB1}
	\E \hat{\mu}^2 \trace(\bR_u) \left[ \zeta + \sigma_v^2 \right]
		+ \trace(\bQ) = 2 \E \hat{\mu} \zeta
		\text{.}
\end{equation}
Recall that~$\zeta$ is the steady-state value of the EMSE~$\zeta(i)$. All that remains to obtain~\eqref{E:2LMSFB} is to determine the supervisor moments needed to evaluate~\eqref{E:2LMSFB1}.

\subsection{Supervisor steady-state analysis}
	\label{S:SSSupervisor}

To analyze the steady-state value of the supervising parameter we adopt the following typical assumption~\cite{Bershad08a, Candido10t, Kozat11t, Jeronimo06m, Magno08i}:

\begin{assumption}[Small variance assumption]\label{A:SSSupervisorVar}

The variance of $\eta(i)$ becomes negligible as $i \to \infty$, so that $\E \eta^2(i) \approx [\E \eta(i)]^2$.

\end{assumption}

\noindent Immediately, \eqref{E:2LMSFB1} becomes
\begin{equation}\label{E:SSPreEMSE}
	\mubar \trace(\bR_u) \left[ \zeta + \sigma_v^2 \right]
		+ \trace(\bQ) = 2 \mubar \zeta
		\text{,}
\end{equation}
recalling that~$\mubar = \etabar \mu_1 + (1-\etabar) \mu_2$ with~$\etabar = \lim_{i\to\infty} \E \eta(i)$. It is ready that~\eqref{E:SSPreEMSE} can be rearranged as in~\eqref{E:2LMSFB}.

To obtain~$\etabar$, notice from~\eqref{E:GlobalMSD} that the mean supervising parameter must converge to a fixed value for the combination to reach global steady-state, i.e., for~$\E \norm{\tilw_{i}}^2 = \E \norm{\tilw_{i-1}}^2$. In other words, $\etabar$ in~\eqref{E:SSPreEMSE} must be a fixed point of the expected value of~\eqref{E:SupervisorModel}. We characterize this fixed point in the following proposition.

\begin{proposition}\label{T:etabar}

Let~$\eta^o \in \argmin_\eta \zeta$ denote a supervising parameter that minimizes the steady-state global EMSE and~$\eta^\star$ be a fixed point of the expected value of~\eqref{E:SupervisorModel}. Under~A.\ref{A:NoiseIndependence}, A.\ref{A:RandomWalk}, and~A.\ref{A:SSSupervisorSeparation}, if~$\eta^o \in [f(a^{-}),f(a^{+})]$, then~$\eta^\star = \eta^o$.
\end{proposition}

\begin{proof}
See Appendix~\ref{AP:Eta}.
\end{proof}

Proposition~\ref{T:etabar} states that, unless constrained by the activation function, the fixed point of the mean supervising parameter update minimizes the global MSE. Hence, we can approximate the mean steady-state value of the supervising parameters by minimizing the steady-state EMSE expression~\eqref{E:2LMSFB} and projecting the minimum onto the range of~$f$. Similar approximations were used in the steady-state analysis of parallel-independent combinations~\cite{Jeronimo06m, Bershad08a, Candido08a}.

Explicitly, the minimum of~\eqref{E:2LMSFB} is obtained when
\begin{equation}\label{E:preEtao}
\begin{aligned}
	\frac{\del \zeta}{\del \etabar} &= 
	\trace(\bR_u) \sigma_v^2 (\mu_1 - \mu_2)^2 \etabar^2
	\\
	{}&+ \trace(\bR_u) \left[ \trace(\bQ) + 2 \mu_2 \sigma_v^2 \right]
		(\mu_1 - \mu_2) \etabar
	\\
	{}&+ \trace(\bR_u) \sigma_v^2 \mu_2^2
		+ \trace(\bR_u) \trace(\bQ) \mu_2 - \trace(\bQ) = 0
		\text{.}
\end{aligned}
\end{equation}
Taking~$\eta^o$ to be the root of~\eqref{E:preEtao} such that the net step size~$\eta^o \mu_1 + (1-\eta^o) \mu_2 \geq 0$ yields~\eqref{E:SSSupervisor} and concludes the proof of Theorem~\ref{T:SS}.

\section{Transient Performance}
	\label{S:Transient}

The performance analysis from Section~\ref{S:SSTracking} show that coefficients feedback can be used to improve the steady-state and tracking of parallel combinations. Recall, however, that the initial motivation for this topology was addressing the convergence stagnation issue. In this section, we therefore examine the transient behavior of a combination of LMS filters with cyclic coefficients feedback and in doing so, study the impact of cyclic feedback on the supervisor variance. Although transient analyses of convex and affine supervising rules in parallel-independent topologies have been carried out~\cite{Candido10t, Vitor09t, Magno10t, Kozat11t}, the following derivations are valid for generic activation functions by extending the approach from~\cite{Vitor09t}.

In what follows, we consider a stationary system identification scenario, i.e., $\bq_i = 0$ in~\eqref{E:RandomWalkModel}, so that the coefficient error vectors become~$\tilw_{i} = \bw^o - \bw_{i}$ and~$\tilw_{n,i} = \bw^o - \bw_{n,i}$ and the \emph{a priori} errors can be written as~$e_{a}(i) = \bu_i \tilw_{i-1}$ and~$e_{a,n}(i) = \bu_i \tilw_{n,i-1}$. We also need to strengthen some of the assumptions from Section~\ref{S:SSTracking}. Namely,

\begin{assumption}[Input signal assumptions]\label{A:TrData}
$\{ \bu_{i} \}$ is an i.i.d.\ sequence of Gaussian vectors with covariance~$\bR_u$ and independent of~$v(j)$ for all~$i,j$. Consequently, $\{ \bu_{i}, \tilw_{j} \}$, $\{ d(i), d(j) \}$, and $\{ \bu_{i}, d(j) \}$ are independent for $i > j$.
\end{assumption}

\begin{assumption}[Transient supervisor separation principle]\label{A:TrSupervisor}

The supervising parameter varies slowly enough that, for~$m,n = 1,2$,
$\E [ \eta(i-1) e_{a,m}(i) e_{a,n}(i) ] = \E \eta(i-1) \E [ e_{a,m}(i) e_{a,n}(i) ]$.

\end{assumption}

\begin{assumption}[Jointly Gaussian \emph{a priori} errors]\label{A:TrGaussian}

The \emph{a priori} error $\{ e_{a,n}(i) \}$, $n = 1,2$, are zero-mean jointly Gaussian random variables, so that
\begin{align*}
	\E e_{a,n}^4(i) &= 3 \zeta_n^2(i)
	\\
	\E e_{a,1}^{k}(i) e_{a,2}^{\ell}(i) &= 
		\begin{cases}
			0 \text{,} & k + \ell = 3
			\\
			3 \zeta_1(i) \zeta_{12}(i) \text{,} & k = 3, \ell = 1
			\\
			3 \zeta_2(i) \zeta_{12}(i) \text{,} & k = 1, \ell = 3
			\\
			2 \zeta_{12}^2(i) + \zeta_1(i) \zeta_2(i) \text{,} & k = \ell = 2
		\end{cases}
\end{align*}

\end{assumption}

Assumption~A.\ref{A:TrData} is used to evaluate higher-order moments of the input signal. Although stronger than~A.\ref{A:SSData}, it is common in the transient analysis of standalone AFs~\cite{Sayed08a, Diniz13a}. Note that, although the vector sequence~$\{\bu_i\}$ is i.i.d., the elements within each vector~$\bu_i$ can be correlated. Additionally, assumption~A.\ref{A:TrSupervisor} allows us to decouple the analysis of the component filters and the supervisor, and A.\ref{A:TrGaussian} is used to evaluate moments of the \emph{a priori} errors for the supervisor variance recursion.

The following theorem collects the complete transient analysis from this section.

\begin{theorem}\label{T:transient}

Let~$\bR_u = \bU \bLambda \bU^T$ be the eigenvalue decomposition of the regressor covariance matrix. Also, let the component filters coefficient vectors and the supervising parameters be initialized as~$\bw_{n,-1} = \bzero$, $\abar(-1) = 0.5$, and~$\sigma_a^2(-1) = 0$, so that~$\Kbar_{n,-1} = \bU^T \bw^o (\bw^o)^T \bU$ and~$\Delta\Kbar_{n,-1} = \bzero$. Then, under~A.\ref{A:TrData}--A.\ref{A:TrGaussian}, the EMSE at iteration~$i$ can be computed recursively by the following steps
\begin{enumerate}[(i)]
	\item Evaluate the supervising parameter moments using
	\begin{align*}
		\E \eta^2(i-1) &\approx f\left[ \abar(i-1) \right]^2
			+ \sigma_a^2(i-1) f^{\prime}\left[ \abar(i-1) \right]^2
		\\
		\E \eta(i-1) &\approx f\left[ \abar(i-1) \right]
	\end{align*}

	\item Evaluate the global EMSE using~$\zeta_{n}(i) = \trace (\bLambda \Kbar_{n,i-1})$, $\Delta\zeta_{n}(i) = \trace (\bLambda \Delta \Kbar_{n,i-1})$, and
	\begin{align*}
		\zeta(i) &= \left[ \E \eta^2(i-1) \right] [ \Delta\zeta_{1}(i) + \Delta\zeta_{2}(i) ]
		\\
		{}&- 2 \left[ \E \eta(i-1) \right] \Delta\zeta_{2}(i) + \zeta_{2}(i)
	\end{align*}

	\item Update the error covariance matrices:
	$\displaystyle
		\begin{cases}
			\text{\eqref{E:FinalTrMSDFB1},}		& i = rL
			\\
			\text{\eqref{E:FinalTrMSDNoFB1},}	& i \neq rL
		\end{cases}
	$
	
	\item Update the supervisor statistics using
	\begin{align*}
		\bar{a}(i) &\approx \bar{a}(i-1) +
			\mu_a [ (1 - \bar{f}) \Delta\zeta_{2} - \bar{f} \Delta\zeta_{1} ]
			\bar{f}^{\prime}
		\\
		\sigma_a^2(i) &\approx
			\left[ 1 + 2 \mu_a G_1 + \mu_a^2 G_2 \right] \sigma_a^2(i-1)
				+ \mu_a^2 G_v
	\end{align*}
	for the~$G$ defined in~\eqref{E:SupervisorVar}.

\end{enumerate}
For~$m,n = 1,2$ and~$m \neq n$, the error covariance matrices are updated as
\begin{equation}\label{E:FinalTrMSDFB1}
\begin{aligned}
	\Kbar_{n,i} &= \Kbar_{i-1}
		- \mu_n [ \Kbar_{i-1} \bLambda + \bLambda \Kbar_{i-1} ]
		+ \mu_n^2 \bA
	\\
	\Delta \Kbar_{n,i} &= (\mu_m - \mu_n) \left( \Kbar_{i-1} \bLambda
		- \mu_n \bA \right)
\end{aligned}
\end{equation}
for~$i = rL$ and~$\bA = \sigma_v^2 \bLambda + \bLambda \trace (\Kbar_{i-1} \bLambda) + 2 \bLambda \Kbar_{i-1} \bLambda$ or
\begin{equation}\label{E:FinalTrMSDNoFB1}
\begin{aligned}
	\Kbar_{n,i} &= \Kbar_{n,i-1}
		+ \mu_n^2 \bA_n
	\\
	{}&- \mu_n [ \Kbar_{n,i-1} \bLambda + \bLambda \Kbar_{n,i-1} ]
	\\
	\Kbar_{12,i} &= \Kbar_{12,i-1}
		- \mu_1 \bLambda \Kbar_{12,i-1}
	\\
	{}&- \mu_2 \Kbar_{12,i-1} \bLambda
		+ \mu_1 \mu_2 \bA_{12}
		\text{.}
\end{aligned}
\end{equation}
for~$i \neq rL$, $\bA_n = \sigma_v^2 \bLambda + \bLambda \trace(\Kbar_{n,i-1} \bLambda) + 2 \bLambda \Kbar_{n,i-1} \bLambda$, and~$\bA_{12} = \sigma_v^2 \bLambda + 2 \bLambda \Kbar_{12,i-1} \bLambda + \trace(\bLambda \Kbar_{12,i-1}) \bLambda$.

\end{theorem}

Although Theorem~\ref{T:transient} presents our results in full generality, it is more straightforward to interpret these results when the input signal is white.

\begin{corollary}\label{T:transientiid}

For white input signals, i.e., $\bR_u = \sigma_u^2 \bI$, the recursions~\eqref{E:FinalTrMSDFB1} and~\eqref{E:FinalTrMSDNoFB1} reduce to
\begin{equation}\label{E:TrEMSEFBiid}
\begin{aligned}
	\zeta_{n}(i) &= a_n \zeta(i-1) + \mu_n^2 M \sigma_u^4 \sigma_v^2
	\\
	\Delta\zeta_{n}(i) &=
		(\mu_m - \mu_n) \sigma_u^2 \left[ b_n \zeta(i-1) + \mu_n M \sigma_u^2 \sigma_v^2 \right]
\end{aligned}
\end{equation}
for~$i = rL$ and
\begin{equation}\label{E:TrEMSENoFBiid}
\begin{aligned}
	\zeta_n(i) &= a_n \zeta_n(i-1) + \mu_n^2 M \sigma_u^4 \sigma_v^2
	\\
	\Delta\zeta_n(i) &= [1 - \mu_n \sigma_u^2] \Delta\zeta_n(i-1)
	\\
		{}&- \sigma_u^2 b_n [ \mu_n \zeta_n(i-1) - \mu_m \zeta_{12}(i-1) ]
	\\
		{}&+ \mu_n ( \mu_n - \mu_m) M \sigma_u^4 \sigma_v^2
\end{aligned}
\end{equation}
for~$i \neq rL$, $a_n = 1 - 2 \mu_n \sigma_u^2 + \mu_n^2 ( M + 2 ) \sigma_u^4$, and~$b_n = 1 - \mu_n ( M + 2 ) \sigma_u^2$.

\end{corollary}

Theorem~\ref{T:transient} and Corollary~\ref{T:transientiid} show the advantages of cyclic coefficients feedback from the viewpoint of convergence. Firstly, notice from~\eqref{E:TrEMSEFBiid} that the EMSE of both component filters are functions of the global EMSE upon feedback. As such, convergence stagnation is eradicated since the difference between the component errors cannot become too large for moderately sized cycle periods. Indeed, upon feedback, the difference between the~$\zeta_n(i)$ is proportional to the difference between the component filters step sizes. Secondly, coefficients feedback reduces the EMSE/cross-EMSE gap, whose magnitude becomes proportional to~$\abs{\mu_1 - \mu_2}$. As the supervisor analysis suggests, this reduces the supervising parameter variance in~\eqref{E:SupervisorVar}, increasing the stability of the supervisor~(see Section~\ref{S:Sims}). Finally, the expression for the mean supervisor parameters~$\abar$ in step~(iv) reveals why feeding back coefficients at all iterations~($L = 1$) stalls the supervisor adaptation and justifies the use of \emph{cyclic} feedback. This way, the combination can reduce the supervising parameters variance as well as avoid convergence stagnation without hindering the supervisor adaptation.

Before proceeding, note that the supervisor transient analysis given in steps~(i) and~(iv) holds for arbitrary activation functions. Transient models for the convex and affine supervisors from Section~\ref{S:SupervisingRules} can therefore be obtained for appropriate choices of~$f$. For the convex supervisor, using~$f$ as in~\eqref{E:ConvexSupervisor_f} recovers the results from~\cite{Vitor09t}. However, for the affine supervisor, i.e., for~$f(a) = a$, the results in Theorem~\ref{T:transient} differ from the previous literature. Indeed, \cite{Candido08a, Candido10t} use different approximations and~\cite{Kozat11t} relies on a different scheme that does not explicitly evaluate these quantities.


The remainder of this section derives the results from Theorem~\ref{T:transient}, namely the global EMSE expression~\eqref{E:TrGlobalEMSE}~(Section~\ref{S:TrComplete}) and the component filters and supervisor statistics needed to evaluate it~(Sections~\ref{S:TrComponents} and~\ref{S:TrSupervisor}).

\subsection{Global transient analysis}
	\label{S:TrComplete}

As opposed to the steady-state analysis of Section~\ref{S:SSGlobal}, the dynamic behavior of the topology is captured by the transient analysis. Hence, we only need to relate the local EMSEs of the component filters to the global EMSE of the combination. To do so, consider the \emph{a priori} errors relation in~\eqref{E:GlobalAPriori}. Under A.\ref{A:TrSupervisor}, its mean-square value is given by
\begin{equation}\label{E:TrGlobalEMSE}
\begin{aligned}
	\zeta(i) &= \left[ \E \eta^2(i-1) \right] [ \Delta\zeta_{1}(i) + \Delta\zeta_{2}(i) ]
	\\
	{}&- 2 \left[ \E \eta(i-1) \right] \Delta\zeta_{2}(i) + \zeta_{2}(i)
		\text{,}
\end{aligned}
\end{equation}
where~$\Delta\zeta_n(i) = \zeta_n(i) - \zeta_{12}(i)$ for~$n = 1,2$. We immediately obtain step~(ii) of Theorem~\ref{T:transient}. In the sequel, we derive recursions for the component filters EMSEs/cross-EMSEs required to evaluate~\eqref{E:TrGlobalEMSE}.

\subsection{Component filters transient analysis}
	\label{S:TrComponents}

Start by recalling that~$\{\bu_i,\tilw_{i-1}\}$ are independent under A.\ref{A:TrData}, so that
\begin{equation}\label{E:TrEMSE}
\begin{aligned}
	\zeta_{n}(i) = \E \bu_i \tilw_{n,i-1} \tilw_{n,i-1}^T \bu_i^T = \trace (\bR_u \bK_{n,i-1})
	\\
	\zeta_{12}(i) = \E \bu_i \tilw_{1,i-1} \tilw_{2,i-1}^T \bu_i^T = \trace (\bR_u \bK_{12,i-1})
\end{aligned}
\end{equation}
where~$\bK_{n,i} = \E \tilw_{m,i}^{} \tilw_{n,i}^T$ and~$\bK_{12,i} = \E \tilw_{1,i}^{} \tilw_{2,i}^T$ are the covariance matrices of the coefficient error vectors. From the linearity of the trace, it is ready that~$\Delta\zeta_n(i) = \trace (\bR_u \Delta \bK_{n,i-1})$ with~$\Delta \bK_{n,i} = \bK_{n,i} - \bK_{12,i}$. Suffices then to find recursions for these matrices.

To do so, subtract the LMS filter recursion~\eqref{E:LMS} from~$\bw^o$ to get
\begin{equation}\label{E:Trtilw}
	\tilw_{n,i} = \tilw_{n,a} - \mu_n \bu_{i}
		\left[ \bu_i^{T} \tilw_{n,a} + v(i) \right]
		\text{,}
\end{equation}
where~$\tilw_{n,a} = \bw^o - \bw_{n,a}$. From the cyclic feedback definition~\eqref{E:Feedback}, we obtain that~$\bw_{n,a}$ is a function of~$\bw_{n,i-1}$ for~$n = 1,2$. Hence, it is independent of~$\bu_i$ under A.\ref{A:TrData}. The expected value of the outer product and cross-outer product of~\eqref{E:Trtilw} then yields
\begin{equation}\label{E:PreComponentMSD}
\begin{aligned}
	\bK_{n,i} &= \bK_{n,a}
		- \mu_n \bK_{n,a} \bR_u
		- \mu_n \bR_u \bK_{n,a}^T
	\\
	{}&+ \mu_n^2
		\left[ \E \bu_i \bu_i^T \bK_{n,a} \bu_i \bu_i^T
			+ \sigma_v^2 \bR_u \right]
	\\
	\bK_{12,i} &= \bK_{12,a}
		- \mu_2 \bK_{12,a} \bR_u
		- \mu_1 \bR_u \bK_{12,a}^T
	\\
	{}&+ \mu_1 \mu_2
		\left[ \E \bu_i \bu_i^T \bK_{12,a} \bu_i \bu_i^T
			+ \sigma_v^2 \bR_u \right]
\end{aligned}
\end{equation}
where~$\bK_{n,a} = \E \tilw_{n,a}^{} \tilw_{n,a}^T$ and~~$\bK_{12,a} = \E \tilw_{1,a}^{} \tilw_{2,a}^T$. Proceeding in a similar fashion, we obtain from~\eqref{E:Feedback} that
\begin{equation}\label{E:PrePriorMSD}
\begin{aligned}
	\bK_{n,a} &= \delta(i - rL) \bK_{i-1} +
		\left[ 1 - \delta(i - rL) \right] \bK_{n,i-1}
	\\
	\bK_{12,a} &= \delta(i - rL) \bK_{i-1} +
		\left[ 1 - \delta(i - rL) \right] \bK_{12,i-1}
\end{aligned}
\end{equation}

The only thing left to evaluate in~\eqref{E:PreComponentMSD} is the fourth-order moment of the regressor. To obtain a closed-form expression for these statistics, we first whiten the regressor vector using the eigenvalue decomposition~$\bR_u = \bU \bLambda \bU^T$. Applying the similarity transformation~$\bU$ to~\eqref{E:PreComponentMSD} and~\eqref{E:PrePriorMSD} then gives
\begin{equation}\label{E:ComponentMSD1}
\begin{aligned}
	\Kbar_{n,i} &= \Kbar_{n,a}
		- \mu_n \Kbar_{n,a} \bLambda
		- \mu_n \bLambda \Kbar_{n,a}^T
	\\
	{}&+ \mu_n^2
		\left[ \E \ubar_i \ubar_i^T \Kbar_{n,a} \ubar_i \ubar_i^T
			+ \sigma_v^2 \bLambda \right]
	\\
	\Kbar_{12,i} &= \Kbar_{12,a}
		- \mu_2 \Kbar_{12,a} \bLambda
		- \mu_1 \bLambda \Kbar_{12,a}^T
	\\
	{}&+ \mu_1 \mu_2
		\left[ \E \ubar_i \ubar_i^T \Kbar_{12,a} \ubar_i \ubar_i^T
			+ \sigma_v^2 \bLambda \right]
\end{aligned}
\end{equation}
\begin{equation}\label{E:PriorMSD}
\begin{aligned}
	\Kbar_{n,a} &= \delta(i - rL) \Kbar_{i-1} +
		\left[ 1 - \delta(i - rL) \right] \Kbar_{n,i-1}
	\\
	\Kbar_{12,a} &= \delta(i - rL) \Kbar_{i-1} +
		\left[ 1 - \delta(i - rL) \right] \Kbar_{12,i-1}
\end{aligned}
\end{equation}
where~$\ubar_i^T = \bu_i^T \bU$ is the whitened version of~$\bu_i$ and for any matrix~$\bH$: $\bm{{\hat{\bH}}} = \bU^T \bH \bU$. Given that the trace is invariant to similarity transformations, the EMSE and cross-EMSE are recovered using
\begin{equation}\label{E:TrEMSECorr}
\begin{aligned}
	\zeta_{mn}(i) &= \trace (\bLambda \Kbar_{mn,i-1})
	\\
	\Delta\zeta_{n}(i) &= \trace (\bLambda \Delta \Kbar_{n,i-1})
\end{aligned}
\end{equation}
Since~$\ubar_i$ is uncorrelated, i.e., $\E \ubar_i^T \ubar_i = \bLambda$, we can use the fourth-order relation from~\cite[Lemma~A.2]{Sayed08a} to write~\eqref{E:ComponentMSD1} as
\begin{equation}\label{E:ComponentMSD}
\begin{aligned}
	\Kbar_{n,i} &= \Kbar_{n,a}
		- \mu_n \Kbar_{n,a} \bLambda
		- \mu_n \bLambda \Kbar_{n,a}^T
	\\
		{}&+ \mu_n^2 \left[ \bLambda \trace (\Kbar_{n,a} \bLambda) + 2 \bLambda \Kbar_{n,a} \bLambda + \sigma_v^2 \bLambda \right]
	\\
	\Kbar_{12,i} &= \Kbar_{12,a}
		- \mu_2 \Kbar_{12,a} \bLambda
		- \mu_1 \bLambda \Kbar_{12,a}^T
	\\
		{}&+ \mu_1 \mu_2
			\left[ \bLambda \trace (\Kbar_{12,a} \bLambda) + 2 \bLambda \Kbar_{12,a} \bLambda + \sigma_v^2 \bLambda \right]
\end{aligned}
\end{equation}

The results in step~(iii) of Theorem~\ref{T:transient} are obtained from~\eqref{E:ComponentMSD} by either using~\eqref{E:PriorMSD} with~$i = rL$, yielding~\eqref{E:FinalTrMSDFB1}, or with~$i \neq rL$, yielding~\eqref{E:FinalTrMSDNoFB1}. We are now only missing recursions for the supervisor statistics to complete our analysis.

\subsection{Supervisor transient analysis}
	\label{S:TrSupervisor}

To carry out the supervisor analysis for general activation functions, we use a linearization argument similar to~\cite{Vitor09t}. First, we approximate the activation function~$f$ in~\eqref{E:SupervisorModel} by its first order Taylor expansion around the mean auxiliary parameter~$\abar(i) = \E a(i)$ as in
\begin{equation}\label{E:fTaylor}
	\eta(i) = f[a(i)] \approx f[\abar(i)]
		+ f^\prime[\abar(i)] \cdot \left[ a(i) - \abar (i) \right]
		\text{.}
\end{equation}
Once again, we write~$f^\prime$ for the derivative of~$f$. Note that the expected value of the second term of~\eqref{E:fTaylor} is zero. Hence, we can evaluate the moments of~$\eta$ in~\eqref{E:TrGlobalEMSE} as
\begin{equation}\label{E:SupervisorMoments}
\begin{aligned}
	\E \eta^2(i) &\approx f^2[\abar(i)]
		+ \sigma_a^2(i) {f^{\prime}}^2[\abar(i)]
	\\
	\E \eta(i) &\approx f[\abar(i)]
		\text{,}
\end{aligned}
\end{equation}
where~$\sigma_a^2(i) = \E \left[ a(i) - \abar(i) \right]^2$ denotes the variance of the auxiliary parameter~$a$. The expressions in~\eqref{E:SupervisorMoments} give step~(i) of Theorem~\ref{T:transient}. Suffice now to find recursions for~$\abar(i)$ and~$\sigma_a^2(i)$.

To do so, we write the supervisor model~\eqref{E:adaptA} in terms of the component filters \emph{a priori} errors. Explicitly, using the \emph{a priori} error relations~\eqref{E:GlobalAPriori}--\eqref{E:estimationAPrioriError}, we obtain
\begin{multline}\label{E:SupervisorModel2}
	a(i) = a(i-1) + \mu_{a} \left\{
		[ 1 - f_{i-1} ] f_{i-1}^{\prime} e_{a,2}^2(i)
	\vphantom{\sum}\right.\\\left.
	{}+ [ 2 f_{i-1} - 1 ] f_{i-1}^{\prime} e_{a,1}(i) e_{a,2}(i)
		- f_{i-1} f_{i-1}^{\prime} e_{a,1}^2(i)
	\right.\\\left.\vphantom{\sum}
	{}+ f_{i-1}^{\prime} [ e_{a,2}(i) - e_{a,1}(i) ] v(i) \right\} 
		\text{,}
\end{multline}
where we write~$f_{i-1} = f[a(i-1)]$ and~$f_{i-1}^\prime = f^\prime[a(i-1)]$ for conciseness. We express~$\eta(i-1)$ as~$f[a(i-1)]$ in~\eqref{E:SupervisorModel2} to stress its dependence on the activation function. We can then use the linearization~\eqref{E:fTaylor} to obtain
\begin{equation}\label{E:ATaylor}
\begin{aligned}
	a(i) &\approx a(i-1) + \mu_{a} \left[ F_2 e_{a,2}^2(i)
		+ F_{12} e_{a,1}(i) e_{a,2}(i)
	\vphantom{\sum}\right.
	\\
	&\left.\vphantom{\sum}
	{}- F_1 e_{a,1}^2(i) + F_v [ e_{a,2}(i) - e_{a,1}(i) ] v(i)
	\right]
\end{aligned}
\end{equation}
with
\begin{align*}
	F_1 &\approx \fbar \fpbar
		+ \left[ (\fpbar)^{2} + \fbar \fppbar \right]
	\left[ a(i-1) - \abar \right]
	\\
	F_2 &\approx (1-\fbar) \fpbar
		+ \left[ -(\fpbar)^{2} + (1-\fbar) \fppbar \right]
	\left[ a(i-1) - \abar \right]
	\\
	F_{12} &\approx (2\fbar - 1) \fpbar
		+ \left[ 2 (\fpbar)^{2} + (2\fbar - 1) \fppbar \right]
	\left[ a(i-1) - \abar \right]
	\\
	F_v &\approx \fpbar
		+ \fppbar \left[ a(i-1) - \abar \right]
\end{align*}
where~$\fbar = f(\abar)$, $\fpbar = f^\prime(\abar)$, and~$\fppbar = f^{\prime\prime}(\abar)$, for~$f^{\prime\prime}$ denoting the second derivative of~$f$. The index on all~$\abar(i-1)$ was omitted for clarity.

Using A.\ref{A:TrSupervisor}, we can separate the expected value of the \emph{a priori} errors and the~$F$ in~\eqref{E:ATaylor}. Then, by noticing that~$\E [ a(i-1) - \abar ] = 0$ we obtain
\begin{equation}\label{E:SupervisorMean}
	\abar(i) \approx \abar(i-1)
		+ \mu_a \left[ (1 - \fbar) \Delta\zeta_{2}
			- \fbar \Delta\zeta_{1} \right] \fpbar
		\text{.}
\end{equation}
The recursion for the variance of~$a$ also requires~A.\ref{A:TrGaussian} to evaluate the higher-order moments of the \emph{a priori} errors that appear when taking the expected value of the square of~\eqref{E:ATaylor}. Then, algebraic manipulations give
\begin{equation}\label{E:SupervisorVar}
	\sigma_a^2(i) \approx [ 1 + 2 \mu_a G_1 + \mu_a^2 G_2 ]
		\sigma_a^2(i-1) + \mu_a^2 G_v
\end{equation}
where
\begin{align*}
	G_{1} &= [ (1 - \fbar) \Delta\zeta_{2} - \fbar \Delta\zeta_{1} ]
		\fppbar
	- [ \Delta\zeta_{1} + \Delta\zeta_{2} ] ( \fpbar )^2
	\\
	G_{2} &= 3 ( \fpbar )^4 ( \Delta\zeta_1 + \Delta\zeta_2 )^2
	+ ( \fppbar )^2
		[ \zeta_2 ( \Delta\zeta_1 + \Delta\zeta_2 ) + 2 \Delta\zeta_2^2 ]
	\\
	{}&+ 3 \fbar ( \fppbar )^2 ( \Delta\zeta_1 + \Delta\zeta_2 )
		[ \fbar \Delta\zeta_1 - (2 - \fbar) \Delta\zeta_2 ]
	\\
	{}&+ 6 ( \fpbar )^2 \fppbar ( \Delta\zeta_1 + \Delta\zeta_2 )
		[ \fbar \Delta\zeta_1 - (1 - \fbar) \Delta\zeta_2 ]
	\\
	{}&+ ( \fppbar )^2 ( \Delta\zeta_1 + \Delta\zeta_2 ) \sigma_v^2
	\\
	G_{v} &= ( \fpbar )^2
	\left\{
		\Delta\zeta_2^2 + ( \Delta\zeta_1 + \Delta\zeta_2 ) \times{}
	\vphantom{\sum}\right.
	\\
	&\left.\qquad\qquad\vphantom{\sum}
	\left[
		2 \fbar^2 ( \Delta\zeta_1 + \Delta\zeta_2 )
		- 4 \fbar \Delta\zeta_2 + \zeta_2 + \sigma_v^2
	\right] \right\} 
\end{align*}
Step~(iv) of Theorem~\ref{T:transient} is given by~\eqref{E:SupervisorMean} and~\eqref{E:SupervisorVar}. This completes the proof of Theorem~\ref{T:transient}.

\begin{figure}[t]
	\centering
	\includesvg{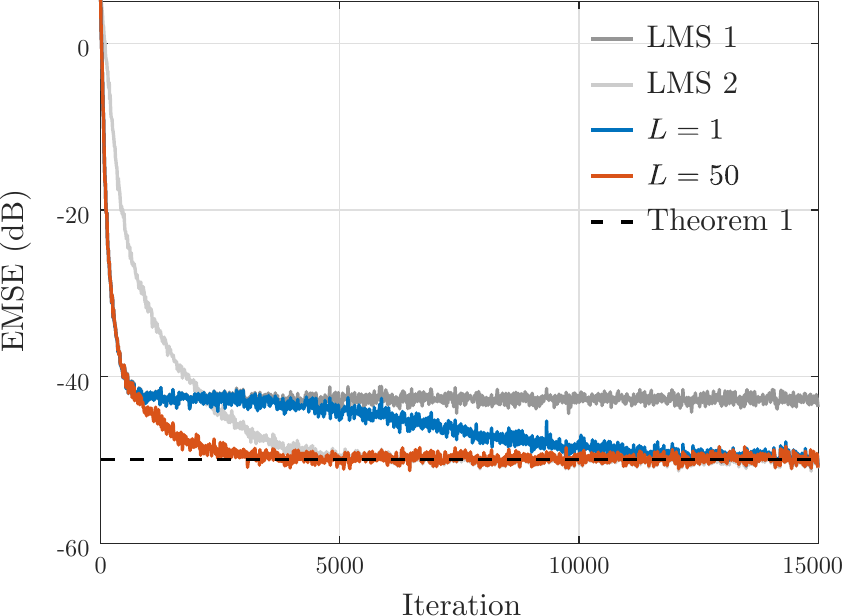}
	\caption{Steady-state analysis of a combination of LMS filters with coefficients feedback. \textbf{Correlated stationary scenario}~(see Section~\ref{S:Sims}): $M = 10$, $\sigma_u^2 = 1$, $\sigma_v^2 = 10^{-3}$,
	$\gamma = 0.7$, $\mu_1 = 0.01$, $\mu_2 = 0.002$;
	\textbf{Cyclic coefficients feedback}: $L = 1$,
		$\tilde{\mu} = 30$, $\beta = 0.9$,
		and~$\eps = 10^{-2}$~(normalized convex supervisor);
	\bm{$L = 50$}: $\tilde{\mu} = 10$, $\beta = 0.9$, $\eps = 10^{-3}$,
		(normalized convex supervisor).}
		\label{F:2LMSStat}
\end{figure}

\begin{figure}[t]
	\centering
	\includesvg{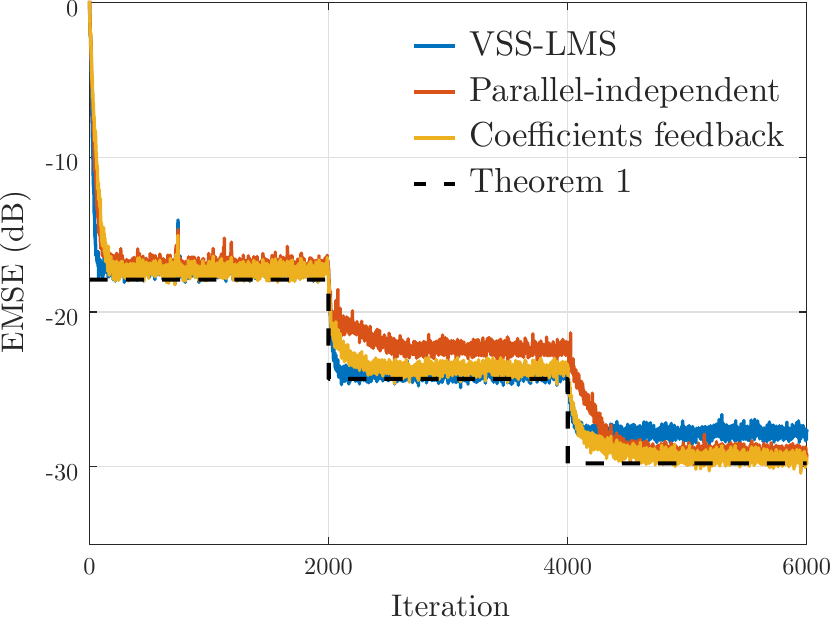}
	\caption{Tracking analysis of combinations with coefficients feedback.
		\textbf{White nonstationary scenario}~(see Section~\ref{S:Sims}):
		$M = 10$, $\sigma_u^2 = 1$, $\sigma_v^2 = 10^{-2}$,
		$\mu_1 = 0.08$, and $\mu_2 = 0.005$;
		$\sigma_q^2 = 10^{-4}$ until~$i = 2000$,
		then~$\sigma_q^2 \to 10^{-5}$ until~$i = 4000$,
		then~$\sigma_q^2 \to 10^{-6}$.
		\textbf{Parallel-independent}~($L \to \infty$):
			$\tilde{\mu} = 0.5$, $\beta = 0.9$, and
			$\eps = 10^{-2}$~(normalized convex supervisor);
		\textbf{Cyclic coefficients feedback}: $L = 10$,
			$\tilde{\mu} = 0.7$, $\beta = 0.7$,
			and~$\eps = 10^{-2}$~(normalized convex supervisor);
		\textbf{VSS-LMS}~\cite{Kwong92v}:
			$\mu(i) = 0.95 \cdot \mu(i-1) + 10^{-1} \cdot e(i)^2$,
			$\mu_{\text{max}} = \mu_1$, and $\mu_{\text{min}} = \mu_2$.}
		\label{F:2LMSNonStat}
\end{figure}

\section{Simulations}
	\label{S:Sims}

The numerical examples in this section illustrate both the performance of combinations with coefficients feedback and the theoretical results from the previous sections. They follow the data model introduced in Section~\ref{S:SSTracking}, taking the regressor~$\bu_i$ to be a delay line that captures samples~$u(i)$ from a zero-mean Gaussian process with~$\sigma_u^2 = \E u(i)^2$. For simulations with \emph{white input data}, the~$u(i)$ are taken to be zero-mean Gaussian i.i.d{.} random variable. For simulations with \emph{correlated input data}, the~$u(i)$ follow the first-order autoregressive model
\begin{equation}
	u(i) = \gamma u(i-1) + \sqrt{1 - \gamma^2} \, x(i)
		\text{,}
\end{equation}
where~$x(i)$ is a zero-mean Gaussian i.i.d{.} random variable with variance~$\sigma_u^2$ and~$0 < \gamma < 1$. Unless stated otherwise, all curves are ensemble averages of 300 independent realizations.

\begin{figure}[t]
	\centering
	\includesvg{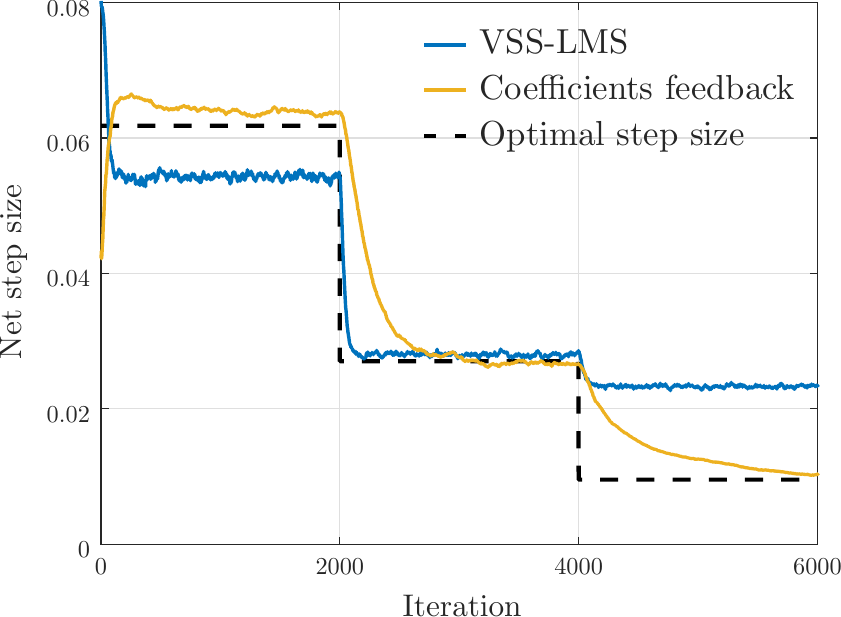}
	\caption{Tracking analysis of combinations with coefficients feedback: equivalent step size. Same setting as Fig.~\ref{F:2LMSNonStat}.}
		\label{F:2LMSNonStat2}
\end{figure}

\begin{figure*}[t]
	\centering
	\includesvg{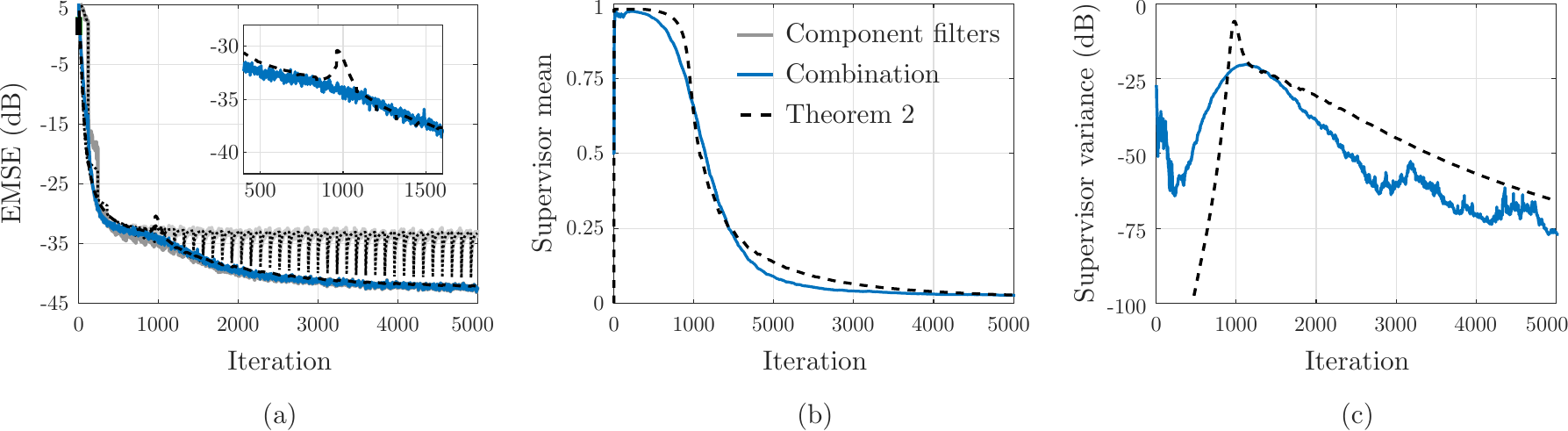}
	\caption{Transient analysis of a combination of LMS filters with coefficients feedback: convex supervisor.
	\textbf{Correlated stationary scenario}~(see Section~\ref{S:Sims}):
		$M = 10$, $\sigma_u^2 = 1$, $\sigma_v^2 = 10^{-2}$,
		$\gamma = 0.7$, $\mu_1 = 0.05$, $\mu_2 = 0.005$,
		$L = 100$, and~$\mu_a = 250$~(convex supervisor).
		Ensemble averages of~$1000$ realizations.}
		\label{F:Tr2LMSConvexSupervisor}
\end{figure*}

\begin{figure*}[t]
	\centering
	\includesvg{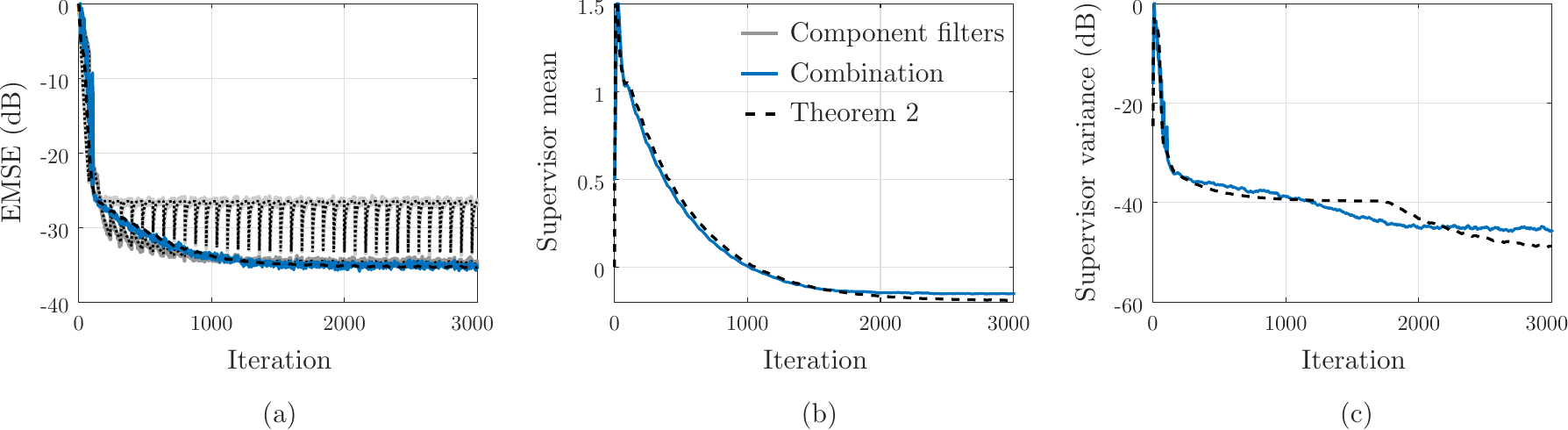}
	\caption{Transient analysis of a combination of LMS filters with coefficients feedback: affine supervisor.
	\textbf{White stationary scenario}~(see Section~\ref{S:Sims}):
		$M = 7$, $\sigma_u^2 = 1$, $\sigma_v^2 = 10^{-2}$,
		$\mu_1 = 0.05$, $\mu_2 = 0.01$, $L = 80$,
		$\mu_{\eta} = 2$, and~$\eta \geq -0.25$.
		Ensemble averages of~$1000$ realizations.}
		\label{F:Tr2LMSAffineSupervisor}
\end{figure*}

\subsection{Steady-state and tracking performance}
\label{S:SimsSSTracking}

For a stationary system, i.e., for~$\bq_i = 0$ in~\eqref{E:RandomWalkModel}, Fig.~\ref{F:2LMSStat} compares the steady-state results from Theorem~\eqref{T:SS} to numerical simulations using a convex supervisor. Given that the supervising parameter is constrained to~$[0,1]$, the performance models yield accurate predictions. As discussed in Section~\ref{S:SSGlobal}, if~$\eta$ is unconstrained, coefficients feedback drives the net step size close to zero as in Fig.~\ref{F:StatAffine2}. Since Theorem~\ref{T:SS} relies on the assumption that the supervisor has small variance, its estimate of~$\etabar$ is no longer reliable. Nevertheless, \eqref{E:2LMSFB} can still be used to estimate the steady-state EMSE for specific values of~$\etabar$. Moreover, the transient analysis from Theorem~\ref{T:transient} can be used in this case~(see Section~\ref{S:SimsTransient}).

To showcase the tracking capabilities of the coefficients feedback topology, let~$\bq_i$ in~\eqref{E:RandomWalkModel} be a zero-mean Gaussian random variable with covariance~$\bQ = \sigma_q^2 \bI$. Recall from Theorem~\ref{T:SS} that for small~$L$ its steady-state error is equivalent to that of a standalone LMS filter with step size~$\mubar$, an affine/convex combination of the component filters step sizes. For illustration, we therefore compare its performance to the VSS-LMS filter from~\cite{Kwong92v} in Fig.~\ref{F:2LMSNonStat}. The VSS-LMS parameters were designed to give good performance in the first, most stringent scenario~($\sigma_q^2 = 10^{-4}$). We also include the results for a parallel-independent combination.

Although the performance of both combinations and the VSS-LMS filter are comparable in the first scenario, the same does not hold as the nonstationarity level changes. In the two other scenarios~($\sigma_q^2 = 10^{-5}$ and~$\sigma_q^2 = 10^{-6}$), either the parallel-independent or the VSS-LMS has higher misadjustment. The parallel-independent topology performs well in the first and last scenarios because the step size of each component filter was chosen close to the optimal~$\mu^o$ in these cases~(see Fig.~\ref{F:2LMSNonStat2}). However, it performs worse when this is not the case. The VSS-LMS algorithm, although effective, is less robust than the normalized convex supervisor used by the combinations~\cite{Jeronimo16c}. Thus, its performance is sensitive to the choice of parameters and it cannot track the optimal step size in the last scenario~(Fig.~\ref{F:2LMSNonStat2}).

\subsection{Transient performance}
\label{S:SimsTransient}

Figs.~\ref{F:CyclicComponents}b and~\ref{F:2LMSStat} already showcased the effectiveness of coefficients feedback in addressing the convergence stagnation issue. Figs.~\ref{F:Tr2LMSConvexSupervisor} and~\ref{F:Tr2LMSAffineSupervisor} additionally illustrate the results from Theorem~\ref{T:transient} for both convex and affine supervisors. Overall, the transient analysis matches the simulations. Small deviations occur in the convex supervisor case when the combination switches between component filters~(see detail in Fig.~\ref{F:Tr2LMSConvexSupervisor}a). This is due to the first-order approximation used to derive the supervisor variance recursion in Section~\ref{S:TrSupervisor}. This phenomenon is not as apparent in the affine supervisor simulations because~\eqref{E:SupervisorMoments}--\eqref{E:ATaylor} are exact when the activation function is linear~(Fig.~\ref{F:Tr2LMSAffineSupervisor}a). Still, the supervisor analyses show good agreement and display the same trends as the simulations~(Figures~\ref{F:Tr2LMSConvexSupervisor}b--c and~\ref{F:Tr2LMSAffineSupervisor}b--c). Despite the cyclostationary behavior of the component filters due to coefficients feedback, notice that the mean supervising parameters converge to a constant steady-state regime. This observation was used to provide the asymptotic performance results from Theorem~\ref{T:SS} without relying on a full transient analysis.

\section{Conclusion}

This paper proposed a new structure for the combination of adaptive filters by making use of cyclic coefficients feedback. This novel topology is able to theoretically describe standalone VSS adaptive algorithms as instances~($L = 1$) of this combination. Steady-state, tracking, and transient analyses showed how cyclic coefficients feedback can improve performance of combinations by addressing the convergence stagnation issue, improving tracking misadjustment, and reducing the supervising parameter variance. Numerical simulations illustrated the good fit of the derived models and showed that existing parallel combinations can be effectively and efficiently improved using the techniques described in this work. This topology opens up new applications for combinations of AFs, such as complexity reduction~\cite{Chamon14t} and rescue techniques~\cite{Sayed08a}.

\appendices

\section{Proof of Proposition~\ref{T:VarCovar}}
	\label{AP:VarCovar}

\begin{proof}

From the cyclic feedback relation~\eqref{E:Feedback}, observe that~$L = 1 \Rightarrow \bw_{n,a} = \bw_{i-1}$ for all~$i$. Thus, the update of the LMS component filters~\eqref{E:LMS} becomes
\begin{equation}\label{E:AFModelFB}
	\bw_{n,i} = \bw_{i-1} + \mu_{n} \bu_i e(i)
		\text{,} \quad n = 1,2 \text{,}
\end{equation}
where we recall that~$e(i) = d(i) - \bu_i^{T} \bw_{i-1}$ is the global output estimation error. The coefficients error vector statistics in~\eqref{E:PreFB} are obtained by first subtracting~\eqref{E:AFModelFB} from~$\bw^o_{i}$ to get
\begin{equation*}
	\tilw_{n,i} = \bw^o_i - \bw_{i-1} - \mu_{n} \bu_i e(i)
		\text{,} \quad n = 1,2 \text{.}
\end{equation*}
Then, notice that the random walk system model from~\eqref{E:RandomWalkModel} implies that~$\bw^o_i - \bw_{i-1} = \tilw_{i-1} + \bq_i$. Hence, we can write
\begin{equation}\label{E:CoefErrorFB}
	\tilw_{n,i} = \tilw_{i-1} + \bq_i - \mu_{n} \bu_i e(i)
		\text{,} \quad n = 1,2 \text{.}
\end{equation}
Taking the expectation of the appropriate inner products of~\eqref{E:CoefErrorFB} results in
\begin{subequations}\label{E:PreFB1}
\begin{align}
	\E \norm{\tilw_{n,i}}^2 &= \E \norm{\tilw_{i-1}}^2 + \trace(\bQ)
	\notag\\
	{}&- 2 \mu_{n} \E e_a(i) e(i)
		+ \mu_{n}^2 \E \norm{\bu_i}^2 e(i)^2
	\\
	\E \tilw_{1,i}^T \tilw_{2,i} &= \E \norm{\tilw_{i-1}}^2 + \trace(\bQ)
	\notag\\
	{}&- (\mu_1 + \mu_2) \E e_a(i) e(i)
	\\\notag
	{}&+ \mu_1 \mu_2 \E \norm{\bu_i}^2  e(i)^2
		\text{.}
\end{align}
\end{subequations}
where we used the fact that~$\E \norm{\bq_i}^2 = \trace(\bQ)$ and that all terms linear in~$\bq_i$ vanish due to A.\ref{A:RandomWalk}.

Note that under~A.\ref{A:SSData}, it holds that~$\E \norm{\bu_i}^2 e(i)^2 = \trace(\bR_u) \E e(i)^2$ and that~A.\ref{A:NoiseIndependence} implies that~$\E e(i)^2 = \zeta(i) + \sigma_v^2$ and~$\E e_a(i) e(i) = \zeta(i)$. Using these relations,~\eqref{E:PreFB1} can written as~\eqref{E:PreFB}.
\end{proof}

\section{Proof of Proposition~\ref{T:etabar}}
	\label{AP:Eta}
	
\begin{proof}

We start by evaluating the global EMSE minimizer~$\eta^o$. Fix~$\eta$ in the \emph{a priori} error relation~\eqref{E:GlobalAPriori} to obtain
\begin{equation}\label{E:preEMSE1}
	\zeta(i) = \E \left[ \eta e_{1,a}(i) + (1-\eta) e_{2,a}(i) \right]^2
\end{equation}
by recalling that~$\zeta(i) = \E e_a(i)^2$. Since~$\eta$ is fixed, we can expand~\eqref{E:preEMSE1} to get, as $i \to \infty$,
\begin{equation}\label{E:preEMSE2}
	\zeta = \eta^2 \zeta_{1} + 2 \eta (1-\eta) \zeta_{12}
		+ (1-\eta)^2 \zeta_{2}
		\text{.}
\end{equation}
We can find~$\eta^o$ by setting the derivative of~\eqref{E:preEMSE2} to zero. Explicitly,
\begin{equation*}
	\frac{\del \zeta}{\del \eta} = 0 \Leftrightarrow
		\eta \Delta\zeta_{1} - (1-\eta) \Delta\zeta_{2}(i) = 0
		\text{,}
\end{equation*}
where~$\Delta\zeta_n = \zeta_1 - \zeta_{12}$. Hence,
\begin{equation}\label{E:etao}
	\eta^o = \frac{\Delta\zeta_2}{\Delta\zeta_1 + \Delta\zeta_2}
		\text{.}
\end{equation}

To show that~\eqref{E:etao} is also a fixed point of the mean supervisor update, take the expected value of~\eqref{E:adaptA} under the supervisor separation assumption~A.\ref{A:SSSupervisorSeparation} to get
\begin{equation}\label{E:Ea}
	\abar(i) = \abar(i-1) + \mu_{a} \E e(i) \left[ y_{1}(i) - y_{2}(i) \right]
		\E f^{\prime}[a(i-1)]
		\text{,}
\end{equation}
where~$\abar(i) = \E a(i)$. Any fixed point of~\eqref{E:Ea} is such that
\begin{equation*}
	\mu_{a} \E e(i) \left[ y_{1}(i) - y_{2}(i) \right] \E f^{\prime}[a(i-1)] = 0
\end{equation*}
and since~$\mu_a$ and~$f^\prime$ are strictly positive, this condition reduces to
\begin{equation}\label{E:fixed}
	\E e(i) \left[ y_{1}(i) - y_{2}(i) \right] = 0
		\text{.}
\end{equation}

From~\eqref{E:fixed}, we obtain an expression for~$\eta^\star$ by expanding the estimation error and the component filters outputs. Explicitly, fixing~$\eta^\star$ in~\eqref{E:GlobalAPriori}--\eqref{E:outputDiff}, we can then rewrite~\eqref{E:fixed} as
\begin{multline}\label{E:fixed2}
	\E \left[ \eta^\star \, e_{a,1}(i) +
		\left( 1-\eta^\star \right) e_{a,2}(i) + v(i) \right]
	\times{}
	\\
	\left[ e_{a,2}(i) - e_{a,1}(i) \right] = 0
		\text{.}
\end{multline}
From the data model assumption~A.\ref{A:NoiseIndependence}, all terms linear in~$v(i)$ vanish, so that as $i \to \infty$, \eqref{E:fixed2} becomes
\begin{equation*}
	(1 - \eta^\star) \zeta_2
	+ (2 \eta^\star - 1) \zeta_{12}
		- \eta^\star \zeta_1 = 0
		\text{,}
\end{equation*}
which can be rearranged as in~\eqref{E:etao}.
\end{proof}

\section*{Acknowledgment}
The authors would like to thank Dr.\ Wilder Bezerra Lopes, Prof.\ Vítor Heloiz Nascimento, and Prof.\ Magno T. M. Silva for fruitful discussions on this topic.

\bibliographystyle{IEEEtran}
\bibliography{IEEEabrv,af,afnet,afcomb,math,sp,telecom}

\begin{thebibliography}{10}
\providecommand{\url}[1]{#1}
\csname url@samestyle\endcsname
\providecommand{\newblock}{\relax}
\providecommand{\bibinfo}[2]{#2}
\providecommand{\BIBentrySTDinterwordspacing}{\spaceskip=0pt\relax}
\providecommand{\BIBentryALTinterwordstretchfactor}{4}
\providecommand{\BIBentryALTinterwordspacing}{\spaceskip=\fontdimen2\font plus
\BIBentryALTinterwordstretchfactor\fontdimen3\font minus
  \fontdimen4\font\relax}
\providecommand{\BIBforeignlanguage}[2]{{%
\expandafter\ifx\csname l@#1\endcsname\relax
\typeout{** WARNING: IEEEtran.bst: No hyphenation pattern has been}%
\typeout{** loaded for the language `#1'. Using the pattern for}%
\typeout{** the default language instead.}%
\else
\language=\csname l@#1\endcsname
\fi
#2}}
\providecommand{\BIBdecl}{\relax}
\BIBdecl

\bibitem{Sayed08a}
A.~Sayed, \emph{Adaptive filters}.\hskip 1em plus 0.5em minus 0.4em\relax
  Wiley-IEEE Press, 2008.

\bibitem{Diniz13a}
P.~Diniz, \emph{Adaptive filtering: {A}lgorithms and practical implementation},
  4th~ed.\hskip 1em plus 0.5em minus 0.4em\relax Springer, 2013.

\bibitem{Harris86v}
R.~Harris, D.~Chabries, and F.~Bishop, ``Variable step ({VS}) adaptive filter
  algorithm,'' \emph{{IEEE} Trans. Signal Process.}, vol.~34, pp. 309--316,
  1986.

\bibitem{Kwong92v}
R.~Kwong and E.~Johnston, ``A variable step size {LMS} algorithm,''
  \emph{{IEEE} Trans. Signal Process.}, vol. 40[7], pp. 1633--1642, 1992.

\bibitem{Mathews93s}
V.~Mathews and Z.~Xie, ``A stochastic gradient adaptive filter with gradient
  adaptive step size,'' \emph{{IEEE} Trans. Signal Process.}, vol. 41[6], pp.
  2075--2087, 1993.

\bibitem{Mayyas95r}
K.~Mayyas and T.~Aboulnasr, ``A robust variable step size {LMS}-type algorithm:
  {A}nalysis and simulations,'' in \emph{ICASSP}, 1995, pp. 1408--1411.

\bibitem{Shin04v}
H.-C. Shin, A.~Sayed, and W.-J. Song, ``Variable step-size {NLMS} and affine
  projection algorithms,'' \emph{{IEEE} Signal Process. Lett.}, vol.~11, pp.
  132--135, 2004.

\bibitem{Singer99u}
A.~Singer and M.~Feder, ``Universal linear prediction by model order
  weighting,'' \emph{{IEEE} Trans. Signal Process.}, vol. 47[10], pp.
  2685--2699, 1999.

\bibitem{Kozat02f}
S.~Kozat and A.~Singer, ``Further results in multistage adaptive filtering,''
  in \emph{ICASSP}, 2002, pp. 1329--1332.

\bibitem{Jeronimo06m}
J.~Arenas-García, A.~Figueiras-Vidal, and A.~Sayed, ``Mean-square performance
  of a convex combination of two adaptive filters,'' \emph{{IEEE} Trans. Signal
  Process.}, vol. 54[3], pp. 1078--1090, 2006.

\bibitem{Chambers06c}
Y.~Zhang and J.~Chambers, ``Convex combination of adaptive filters for a
  variable tap-length {LMS} algorithm,'' \emph{{IEEE} Signal Process. Lett.},
  vol. 13[10], pp. 628--631, 2006.

\bibitem{Bershad08a}
N.~Bershad, J.~Bermudez, and J.-Y. Tourneret, ``An affine combination of two
  {LMS} adaptive filters---transient mean-square analysis,'' \emph{{IEEE}
  Trans. Signal Process.}, vol. 56[5], pp. 1853–--1864, 2008.

\bibitem{Magno08i}
M.~Silva and V.~Nascimento, ``Improving the tracking capability of adaptive
  filters via convex combination,'' \emph{{IEEE} Trans. Signal Process.}, vol.
  56[7], pp. 3137–--3149, 2008.

\bibitem{Cassio07e}
C.~Lopes, E.~Satorius, P.~Estabrook, and A.~Sayed, ``Efficient adaptive carrier
  tracking for {M}ars to {E}arth communications during entry, descent and
  landing,'' in \emph{SSP}, 2007, pp. 517--521.

\bibitem{Jeronimo08a}
J.~Arenas-García and A.~Figueiras-Vidal, ``Adaptive combination of {IPNLMS}
  filters for robust sparse echo cancellation,'' in \emph{MLSP}, 2008, pp.
  221--226.

\bibitem{Cassio10a}
C.~Lopes, E.~Satorius, P.~Estabrook, and A.~Sayed, ``Adaptive carrier tracking
  for {M}ars to {E}arth communications during entry, descent, and landing,''
  \emph{{IEEE} Trans. Aerosp. Electron. Syst.}, vol. 46[4], pp. 1865--1879,
  2010.

\bibitem{Azpicueta11a}
L.~Azpicueta-Ruiz, M.~Zeller, A.~Figueiras-Vidal, J.~Arenas-García, and
  W.~Kellermann, ``Adaptive combination of {V}olterra kernels and its
  application to nonlinear acoustic echo cancellation,'' \emph{{IEEE} Trans.
  Audio, Speech, Signal Process.}, vol. 19[1], pp. 97--110, 2011.

\bibitem{Wilder11i}
W.~Lopes and C.~Lopes, ``Incremental-cooperative strategies in combination of
  adaptive filters,'' in \emph{ICASSP}, 2011, pp. 4132--4135.

\bibitem{Chamon12c}
L.~Chamon, W.~Lopes, and C.~Lopes, ``Combination of adaptive filters with
  coefficients feedback,'' in \emph{ICASSP}, 2012, pp. 3785--3788.

\bibitem{Chamon12d}
L.~Chamon, H.~Ferro, and C.~Lopes, ``A data reusage algorithm based on
  incremental combination of {LMS} filters,'' in \emph{Asilomar Conference on
  Signals, Systems and Computers}, 2012.

\bibitem{Vitor12l}
V.~Nascimento and R.~{de Lamare}, ``A low-complexity strategy for speeding up
  the convergence of convex combinations of adaptive filters,'' in
  \emph{ICASSP}, 2012, pp. 3553--3556.

\bibitem{Ferro14f}
H.~Ferro, L.~Chamon, and C.~Lopes, ``{FIR-IIR} adaptive filters hybrid
  combination,'' \emph{IET Electronics Letters}, vol. 50[7], pp. 501--503,
  2014.

\bibitem{Azpicueta08n}
L.~Azpicueta-Ruiz, A.~Figueiras-Vidal, and J.~Arenas-García, ``A normalized
  adaptation scheme for the convex combination of two adaptive filters,'' in
  \emph{ICASSP}, 2008, pp. 3301--3304.

\bibitem{Candido10t}
R.~Candido, M.~Silva, and V.~Nascimento, ``Transient and steady-state analysis
  of the affine combination of two adaptive filters,'' \emph{{IEEE} Trans.
  Signal Process.}, vol. 58[8], pp. 4064--4078, 2010.

\bibitem{Vitor09t}
V.~Nascimento, M.~Silva, R.~Candido, and J.~Arenas-García, ``A transient
  analysis for the convex combination of adaptive filters,'' in \emph{SSP},
  2009, pp. 53--56.

\bibitem{Kozat10u}
S.~Kozat, A.~Singer, A.~Erdogan, and A.~Sayed, ``Unbiased model combinations
  for adaptive filtering,'' \emph{{IEEE} Trans. Signal Process.}, vol. 58[8],
  pp. 4421--4427, 2010.

\bibitem{Kozat11t}
S.~Kozat, A.~Erdogan, A.~Singer, and A.~Sayed, ``Transient analysis of adaptive
  affine combinations,'' \emph{{IEEE} Trans. Signal Process.}, vol. 59[12], pp.
  6227--6232, 2011.

\bibitem{Jeronimo16c}
J.~{Arenas-García}, L.~{Azpicueta-Ruiz}, M.~Silva, V.~Nascimento, and
  A.~Sayed, ``Combinations of adaptive filters: Performance and convergence
  properties,'' \emph{{IEEE} Signal Process. Mag.}, vol. 33[1], pp. 120--140,
  2016.

\bibitem{Gredilla10a}
M.~Lazaro-Gredilla, L.~Azpicueta-Ruiz, A.~Figueiras-Vidal, and
  J.~Arenas-García, ``Adaptively biasing the weights of adaptive filters,''
  \emph{{IEEE} Trans. Signal Process.}, vol. 58[7], pp. 3890--3895, 2010.

\bibitem{Vitor10t}
V.~Nascimento, M.~Silva, L.~Azpicueta-Ruiz, and J.~Arenas-Garcia, ``On the
  tracking performance of combinations of least mean squares and recursive
  least squares adaptive filters,'' in \emph{IEEE International Conference on
  Acoustics Speech and Signal Processing}, 2010, pp. 3710--3713.

\bibitem{Huang06a}
Y.~Huang, J.~Benesty, and J.~Chen, \emph{Acoustic {MIMO} Signal
  Processing}.\hskip 1em plus 0.5em minus 0.4em\relax Springer, 2006.

\bibitem{So01c}
H.~So and P.~Ching, ``Comparative study of five {LMS}-based adaptive time delay
  estimators,'' \emph{IEE Proc. Radar, Sonar and Nav.}, vol. 148[1], pp. 9--15,
  2001.

\bibitem{Widrow08a}
B.~Widrow and E.~Walach, \emph{Adaptive inverse control: a signal processing
  approach}.\hskip 1em plus 0.5em minus 0.4em\relax Wiley-IEEE Press, 2008.

\bibitem{Candido08a}
R.~Candido, M.~Silva, and V.~Nascimento, ``Affine combinations of adaptive
  filters,'' in \emph{Asilomar Conference on Signals, Syst. and Comput.}, 2008,
  pp. 236--240.

\bibitem{Magno10t}
M.~Silva, V.~Nascimento, and J.~Arenas-García, ``A transient analysis for the
  convex combination of two adaptive filters with transfer of coefficients,''
  in \emph{ICASSP}, 2010, pp. 3842–--3845.

\bibitem{Chamon14t}
L.~Chamon and C.~Lopes, ``There's plenty of room at the bottom: {I}ncremental
  combinations of sign-error {LMS} filters,'' in \emph{International Conference
  in Acoustic, Speech, and Signal Processing~(ICASSP)}, 2014.

\end{thebibliography}

\end{document}